\documentclass[a4paper,12pt, oneside]{article}
\usepackage{amsmath, amssymb, amsthm,graphicx,xcolor}
\usepackage{xspace,mathrsfs,bbm}
\usepackage[numbers]{natbib}
\usepackage[T1]{fontenc}
\usepackage[theoremfont]{newtxtext}
\usepackage[vvarbb,bigdelims]{newtxmath}

\usepackage{titlesec}

\titleformat{\section}{\center\ff{12}\bfseries}{\thesection.}{.7ex}{}
\titlespacing{\section}{0pc}{4ex plus .5ex minus .5ex}{2ex minus .1ex}

\newcommand{\seq}[3]{(#1_#2)_{#2\in#3}}

\newcommand{\ff}[1]{\fontsize{#1}{#1}\selectfont}

\usepackage{enumitem}
\setlist{noitemsep}
\setlist{nosep}

\newtheorem{theorem}{Theorem}
\newtheorem*{theorem*}{Theorem}

\newtheorem{lemma}[theorem]{Lemma}
\newtheorem{corollary}[theorem]{Corollary}

\theoremstyle{definition}

\newtheorem{examplex}{Example}
\newenvironment{example}
  {\pushQED{\qed}\examplex}
  {\popQED\endexamplex}

\newtheorem{remx}{Remark}
\newenvironment{remark}
  {\pushQED{\qed}\remx}
  {\popQED\endremx}

\definecolor{link}{rgb}{0,0.2,.5}
\definecolor{cite}{rgb}{0,0.3,.8}

\usepackage[normalem]{ulem} 

\newcommand{\changed}[1]{{\color[rgb]{0.8,0,0.4}#1}\marginpar{\ff{8}\color[rgb]{0.8,0,0.4}$\nwarrow$ revised}}

\newcommand{\answ}[1]{{\color[rgb]{0,0.2,0.86}#1}}
\newcommand{\done}{\answ{Done.}}
\usepackage[
pdftitle={},
pdfauthor={},
bookmarks=true,
colorlinks=true,
urlcolor=url,
linkcolor=link,
citecolor=cite,
filecolor=black,
raiselinks=false,
hyperfootnotes=true]{hyperref}

\newcommand{\inv}{^{-1}}
\newcommand{\la}{\lambda}
\newcommand{\al}{\alpha}
\newcommand{\ga}{\gamma}
\newcommand{\cc}{\circ}

\newcommand{\di}[1]{^{(#1)}}

\newcommand{\sq}[1]{\mathbb #1}
\newcommand{\sqNN}{\mathbb Z_+}

\newcommand{\kla}[1]{\left(#1\right)}

\newcommand{\OPR}[1]{{\mathbbm #1}}

\newcommand{\PP}[1]{{\OPR P}(#1)}
\newcommand{\pp}[1]{{\OPR P}(#1)}
\newcommand{\ee}[1]{{\OPR E}#1}
\newcommand{\EE}[1]{{\OPR E}(#1)}

\newcommand{\VAR}[1]{{\rm Var}\kla{#1}}

\newcommand{\GS}{G}
\newcommand{\GG}[1]{\GS_{#1}}
\newcommand{\GEN}[1]{P_{#1}}
\newcommand{\LL}[1]{L_{#1}}

\newcommand{\veps}{\varepsilon}
\newcommand{\lex}[1]{{\itshape #1}}

\newcommand{\proc}[2]{(#1)_{#2}}
\newcommand{\deq}{\sim}
\newcommand{\dto}{\Rightarrow}

\newcommand{\apgf}{a.p.g.f.\xspace}

\newcommand{\pgf}{p.g.f.\xspace}

\newcommand{\noneg}{non-negative\xspace}
\newcommand{\nrv}{\noneg random variable\xspace}
\newcommand{\intval}{\noneg integer-valued\xspace}

\newcommand{\afrac}[2]{#1/#2}
\newcommand{\as}[1]{,\qquad#1}

\newcommand{\given}[1]{|{#1}}
\newcommand{\gv}[2]{#1_#2}
\newcommand{\zo}[1]{{#1}_0}
\newcommand{\II}[1]{I_{#1}}

\newcommand{\bt}[1]{#1^\star}
\newcommand{\ba}[1]{\bt p_{#1}(\al)}

\newcommand{\rad}[1]{R_{#1}}
\newcommand{\mbf}[1]{$\boldsymbol{#1}$}

\newcommand{\DSYM}{{\mathbb M}}
\newcommand{\MS}[1]{\DSYM_{#1}^\star}
\newcommand{\MM}[1]{\DSYM_{#1}}
\newcommand{\MSF}[1]{\DSYM^\star_{#1}(n)}
\newcommand{\MMF}[1]{\DSYM_{#1}(n)}
\newcommand{\WSYM}{{\mathbb B}}
\newcommand{\BB}{\WSYM(n)}
\newcommand{\nors}{\rho}
\newcommand{\nor}[1]{\nors(#1)}

\newcommand{\sig}[1]{\sigma_n(#1)}
\newcommand{\aaa}{\al\inv}

\newcommand{\supp}[1]{{\rm supp}(#1)}

\newcommand{\Cgam}[2]{{\rm Gamma}(#1,#2)}
\newcommand{\Cunif}[2]{{\rm Uniform}[#1,#2]}
\newcommand{\Cexp}[1]{{\rm Exp}(#1)}
\newcommand{\Cpoi}[1]{\mathrm{Poisson}(#1)}
\newcommand{\Cber}[1]{{\rm Bernoulli}(#1)}
\newcommand{\Cbinom}[2]{{\rm Binomial}(#1,#2)}
\newcommand{\Cerl}[2]{{\rm Erlang}(#1,#2)}
\newcommand{\Csibuya}[1]{{\rm Sibuya}(#1)}

\begin{document}
\setlength{\abovedisplayskip}{6pt}
\setlength{\belowdisplayskip}{7pt}
\setlist[enumerate]{itemsep=1pt,topsep=2pt,leftmargin=3em,labelwidth=2em,align=left}
\setlist[enumerate,1]{label=\rm(\roman*)}

\newpage

\setcounter{page}{1}
\title{On Binomial Thinning and Mixing}

\author{Offer Kella\thanks{Department of Statistics and Data Science, the Hebrew University of Jerusalem, Jerusalem 9190501, Israel ({\tt  offer.kella@huji.ac.il}).}
\thanks{Supported in part by grant No. 1647/17 from the Israel Science Foundation and the Vigevani Chair in Statistics.}
\ \ and\ \  Andreas L\"opker\thanks{Department of Computer Science and Mathematics, HTW Dresden, University of Applied Sciences, Friedrich-List-Platz 1, D-01069 Dresden, Germany ({\tt lopker@htw-dresden.de}).}
}
	
\maketitle

\begin{abstract} In this paper we consider the notions of binomial thinning, binomial mixing, their generalizations, certain interplay between them, associated limit theorems and provide various examples.
\vspace{2mm}
 
 
 \vspace{2mm}
 
 \noindent 
{\it AMS Subject Classification (MSC2020):} Primary 60E05, Secondary: 60E10.

\end{abstract}

\section{Introduction}

Given a \intval random variable $Z$ and a probability $\al$, a (binomially) $\alpha$-thinned random variable denoted $\alpha\cc Z$ is a random variable that has, given $Z=n$, a binomial distribution with parameters $n$ and $\alpha$. By definition, a $1$-thinned random variable has the same distribution as the original random variable and a $0$-thinned random variable is distributed like the constant zero.

While, for every $\al\in[0,1]$ and any arbitrary  \intval $Z$, the $\alpha$-thinning operation is well defined, there are \intval random variables that cannot be written as an $\alpha$-thinned `version' of some \intval $Z$ for any $\al<1$. It turns out that for any $X$ there is a minimal $\alpha\in[0,1]$ (possibly one or zero) allowing such a representation. This observation gives rise to a nested hierarchy of subsets of all distributions concentrated on the non-negative integers, parametrized by $\alpha$. We investigate conditions for the membership to these classes, specialize results to distributions with finite support and provide examples.

In addition we also focus on the dual concept of binomial mixtures denoted $W\cc n$, where initially the support of the distribution of $W$ is contained in $[0,1]$.  $W\cc n$ is the notation for a random variable of which conditional distribution given $W$ is binomial with $n$ and $W$ (to be made more precise later). One question that we explore in this context is whether one can perform the same type of binomial mixing operation when the support is not contained in $[0,1]$. We provide a nontrivial example where $W$ has the Gamma distribution and thus has an unbounded support. Another question that arises is regarding the relation between the minimal $\alpha$ mentioned in the previous paragraph with respect to the $\alpha$-thinning operation associated with $W\circ n$, when well defined, and the minimal $\alpha$ such that the distribution of $(W/\alpha)\circ n$ is well defined. It turns out that these two values are equal. 

For both thinning and mixing concepts we provide some limiting results where the limits turn out to have Poisson and mixed Poisson distributions, respectively.

Binomial thinning has become a useful tool when studying count data time series (\cite{al3}, \cite{chw1}, \cite{davis}, or \cite{kang} for a generalisation). Take for example the simple auto-regressive INAR(1) process $\seq{X}{n}{\sq N}$ where
\begin{align*}
X_n=\al\cc X_{n-1}+U_n
\end{align*}
with i.i.d. discrete $\seq{U}{n}{\sq N}$ (serving as noise to the model). 
The binomial thinning operation is also performed in a Galton-Watson branching process with Bernoulli offspring, where each individual survives with probability $\alpha$ and has no descendants. In this context the INAR(1) process relates to the population size if immigration is present. 

One can also think of applications in queueing, e.g. in modified M/G/1 systems with impatient customers (\cite{Coh}, II.4., p.234, for the M/G/1 queue). If $A_n$ denotes the number of customers waiting immediately after the beginning of the $n$th service in an M/G/1 queue, and if customers independently leave the queue with probability $\al$ during that service, then
\begin{align*}
A_{n+1}=(\al\cc A_n+C_n-1)^+
\end{align*}
where $(x)^+=\max\{0,x\}$ and $C_n$ denotes the newly arrived customers (see \cite{box} for such processes).

A somewhat related study to ours has been conducted in \cite{wiuf}, where the emphasis is on subclasses of \intval random variables of the form $\al\cc Z$ with $Z$ \intval. Structural aspects of binomial thinning are also studied in \cite{jk} (the operation is called `dilation' there).

Regarding the mixing operation a prominent mixed distribution is the beta-binomial distribution, finding applications in Bayesian statistics. A reference for this and other distributions appearing in this paper is \cite{kotz}.

The paper is organized as follows. In Section~\ref{Thinning} we discuss the basic $\alpha$-thinning operation and prove various preliminary results. In Section~\ref{MMMS} we give a more in depth study of various concepts introduced in Section~\ref{Thinning}. In Sections~\ref{DBS} and \ref{Examples} we discuss the case where the support of the random variables is bounded and provide various examples. In particular we give a necessary and sufficient characterization as to when an $X$ does not have an $\alpha$-thinned distribution for any $\alpha<1$. Finally, in Section~\ref{BM} we discuss the concept of binomial mixing and provide two nontrivial examples.

\section{Thinning preliminaries}\label{Thinning}

In what follows, for a random variable $X$, {\em support of $X$} abbreviates {\em support of the distribution of $X$} and {\em a.s.} abbreviates {almost surely} (with respect to the probability discussed). Unless stated otherwise, random variables are assumed to be a.s. finite. The notation `$\sim$' abbreviates `distributed as', `$\dto$' denotes convergence in distribution. We write $\Cgam \alpha\beta$, $\Cexp\la$, $\Cerl m\la$, $\Cunif ab$, $\text{Uniform}(i_1,\ldots,i_m)$, $\Cpoi\la$, $\Cber p$, $\Cbinom n p$ for the gamma/exponential/continuous uniform/discrete uniform distributions with the usual notations for the parameters. Finally, $1_A$ denotes the indicator of a set $A$.

Let ${\MM{1}}$ denote the family of all discrete probability distributions with support contained in $\sqNN=\{0,1,2,\ldots\}$. When we write $X\in\MM 1$ we mean that the distribution of $X$ is in $\MM 1$, or equivalently, that $\pp{X\in\mathbb{Z}_+}=1$.

Given $X\in\MM 1$ and $\al\in[0,1]$ the binomial $\alpha$-thinning operation is defined as
\begin{align}
\al\cc X\sim\sum_{i=1}^X B_i,
\end{align}
where an empty sum (when $X=0$) is defined to be zero and $\proc{B_i}{i\in\sqNN}$ are i.i.d. Bernoulli r.v., independent of $X$, with $\pp{B_i=1}=\al$ (see \cite{steutel}). In particular, note that $0\circ X=0$ a.s. and $1\circ X\sim X$. 

It is simple to show and well known that for independent $X,Y\in\MM{1}$, independent $\al\circ X,\al\circ Y$ and $\al,\beta\in(0,1]$ the following relations hold.
\begin{enumerate}
\item $\al\cc(X+Y)\deq \al\cc X+\al\cc Y$.
\item $\al\cc(\beta\cc X) \deq\beta\cc(\al\cc X)\deq (\al\beta)\cc X$.
\item $\EE{\al\cc X}=\al\ee{X}$ (finite or infinite).
\item  $\VAR{\al\cc X}=\al^2\VAR{X}+\al(1-\al)\ee{X}$, whenever $\ee{ X^2}<\infty$.
\end{enumerate}

The set of all $\alpha$-thinned distributions is denoted by $\MM \alpha$ and when we write $X\in\MM \alpha$ we mean that $X\sim\alpha\circ Z$ for some $Z\in\MM 1$.

For $X\in\MM 1$ we denote the probability generating function (\pgf) of $X$ by
\begin{align*}
\GEN X(s)=\begin{cases}\ee{s^X}&s\in[-1,0)\cup(0,1]\\ P(X=0)&s=0\,,\end{cases}
\end{align*}
which is continuous on $[-1,1]$ and infinitely differentiable on $(-1,1)$. Often it is more convenient to work with the function $\GG{X}(s)=\GEN X(1-s)$, the alternate probability generating
function (\apgf) of $X$, which is continuous on $[0,2]$ and is infinitely differentiable on $(0,2)$. We recall that $n!P(X=n)=P_X^{(n)}(0)=(-1)^nG_X^{(n)}(1)$. One reason to use 
$\GG X(s)$ rather than $P_X(s)$ is the useful relation
\begin{align}\label{basic}
\GG{\al\cc X}(s)=\GG{X}(\al s),
\end{align}
for $s\in[0,1)$, which is straightforward to show ({\em e.g.,} \cite{mc}). In fact, this equality holds on $[0,2]$. The reason is that upon changing the order of summation we have that
\begin{equation}
\sum_{i=0}^\infty |1-s|^i\sum_{k=i}^\infty {k\choose i}\alpha^i(1-\alpha)^{k-i}P(X=i)=\ee (1-\alpha +|1-s|\alpha)^X
\end{equation}
Thus, the sum converges abslutely whenever $1-\alpha+|1-s|\alpha\le 1$ which holds if and only if $s\in[0,2]$. Therefore, it is allowed to change the order of summation when $(1-s)^i$ replaces $|1-s|^i$ which gives $G_X(\alpha s)$ on the right hand side.

A real valued function $f$ is called {\em absolutely monotone} (resp., {\em completely monotone}) on an open interval $(a,b)$ ($a$ or $b$ could be infinite) if it is smooth (=infinitely differentiable) and satisfies $f^{(n)}(t)\ge 0$ (resp., $(-1)^nf^{(n)}(t)\ge 0$) for all $t\in(a,b)$.

We recall that a function $f:(0,1]\to \mathbb{R}$ is a p.g.f. of some, nonnegative, integer valued random variable if and only if it is absolutely monotone on $(0,1)$ and satisfies $f(1-)=f(1)=1$ ({\em e.g.,} Th.~3a, p.~146 of \cite{widder}). Therefore, $G:[0,1)\to\mathbb{R}$ is an a.p.g.f. if and only if it is completely monotone on $(0,1)$ and satisfies $G(0+)=G(0)=1$. 

We note that although $P_X(s)$ (as a function on $[-1,1]$) is absolutely monotone on $(0,1)$ (and hence, also at $0$), it is not necessarily so on $(-1,0)$. Similarly, $G_X(s)$ (as a function on $[0,2]$) need not be completely monotone on $(1,2)$. For a trivial example, take $P(X=1)=1$ and observe that $P_X(s)=s$ is negative on $(-1,0)$ and $G_X(s)=1-s$ is negative on $(1,2)$.

When $Y$ is an a.s. nonnegative random variable we denote by $L_Y(s)=\ee{e^{-sY}}$ the Laplace-Stieltjes transform (LST) of $Y$. We recall that a function $h:[0,\infty)\to\mathbb{R}$ is an LST of some a.s. nonnegative random variable if and only if it is completely monotone on $(0,\infty)$, is continuous (from the right) at zero and satisfies $h(0)=1$. 

For ease of reference, $X$ has a {\em mixed Poisson} distribution if for some nonnegative random variable $Y$, the conditional distribution of $X$ given $Y$ is $\Cpoi Y$, where $\Cpoi 0$ is the
distribution of the constant zero ({\em e.g.,} \cite{gran}). Equivalently, when $X\sim N(Y)$, where $\{N(t)|\,t\ge 0\}$ is a Poisson process with rate $1$ and $Y$ is an independent nonnegative random variable. We note that when $X$ has a mixed Poisson distribution, then $P(X=k)=\ee e^{-Y}Y^k/k!$, $\ee X=\ee Y$ and $
\VAR X=\VAR Y+\ee{Y}$ whenever $\ee{Y^2}<\infty$. It is easy to verify that, since $G_X(s)$ converges absolutely on $[0,2]$, then $G_X(s)=L_Y(s)$ for $s\in[0,2]$. 
Also, if $G_X(s)=L_Y(s)$ on some $S\subset (0,2)$ containing an accumulation point, then since $G_X(\cdot)$ and $L_Y(\cdot)$ are analytic on $(0,2)$ and continuous on $[0,2]$, it follows that $G_X(s)=L_Y(s)$ on $[0,2]$, thus necessarily $X\sim N(Y)$.

\begin{remark}
We note that, although $L_Y(s)$ is finite on $[0,\infty)$ and is equal to $G_{N(Y)}(s)$ on $[0,2]$, it is possible that $G_{N(Y)}(s)$ is infinite or undefined for $s\not\in[0,2]$. One such example is when $Y$ has the density $y^{-2}1_{[1,\infty)}(y)$. It is easy to verify that in this case, letting $S_{n-1}\sim\Cgam{n-1}1$ and recalling $N(1)\sim\Cpoi 1$, then for $n\ge 2$ we have
\begin{equation}
n(n-1)\ee e^{-Y}\frac{Y^n}{n!}=\int_1^{\infty}e^{-y}\frac{y^{n-2}}{(n-2)!}\text{d}y=\pp{S_{n-1}>1}=\pp{N(1)\le n-2}\,.
\end{equation}
Thus, when $s\not\in[0,2]$ we have that $|1-s|>1$, which implies that
\begin{equation}
\lim_{n\to\infty}|1-s|^n\pp{N(Y)=n}=\lim_{n\to\infty}\frac{|1-s|^n}{n(n-1)}\pp{N(1)\le n-2}=\infty
\end{equation} 
and thus 
$\sum_{i=0}^\infty(1-s)^i\pp{N(Y)=i}$ is either infinite (for $s<0$) or undefined (for $s>2$).
\end{remark}

An immediate consequence of the discussion above is the following.

\begin{theorem}\label{ThCM0}  Let $X\in \MM 1$ and $\alpha\in (0,1)$.
\begin{enumerate}
\item There exists $f:[0,1/\alpha)\to \mathbb{R}$ completely monotone on $(0,1/\alpha)$ with $f(0+)=f(0)$  for which $f(s)=G_X(s)$ on $S\subset (0,\min(\alpha^{-1},2))$ containing an accumulation point if and only if $X\in\MM\alpha$. In this case there exists $Z\in\MM 1$ such that $G_X(s)=G_Z(\alpha s)$ on $[0,2]$.
\item There exists $f:[0,\infty)\to\mathbb{R}$ completely monotone on $(0,\infty)$ with $f(0)=f(0+)$ for which $f(s)=G_X(s)$ on $S\subset (0,2)$ containing an accumulation point if and only if $X$ has a mixed Poisson distribution. Equivalently, there exists some nonnegative $Y$ such that $G_X(s)=L_Y(s)$ on $[0,2]$.
\end{enumerate}
\end{theorem}

\begin{proof}
In (i) we note that since $G_X(\cdot)$ and $f(\cdot)$ are analytic on $(0,\alpha^{-1}\wedge 2)$ and continuous on $[0,\alpha^{-1}\wedge 2]$ then $G_X(s)=f(s)$ on $[0,\alpha^{-1}\wedge 2]$.  Also, $f_\alpha(t)=f(t\alpha^{-1})$ is completely monotone on $(0,1)$, continuous at $0$, with $f_\alpha(0)=1$ and thus there exists $Z\in\MM 1$ for which $G_Z(s)=f(s\alpha^{-1})$ for $s\in[0,1)$. Therefore $G_Z(\alpha s)=f(s)$ for $s\in[0,\alpha^{-1})$ and in particular for $s\in[0,1)$ we have $G_Z(\alpha s)=f(s)=G_X(s)$ which is equivalent to $X\sim\alpha\circ Z$, hence $X\in\MM\alpha$. The converse is obvious. The fact that $G_X(s)=G_Z(\alpha s)$ for $s\in[0,2]$ was discussed above. (ii) is also immediate from the discussion above.
\end{proof}

\begin{corollary}\label{CorExt}The following are equivalent.
\begin{enumerate}
\item $X\in \MM \alpha$ for some $\alpha\in(0,1)$.
\item $G_X(s)$ is completely monotone on $(1,c)$ for some $c\in (1,2]$.
\item $P_X(s)$ is absolutely monotone on $(-\epsilon,0)$ for some $\epsilon\in(0,1)$.
\end{enumerate}
\end{corollary}

We note that replacing $(0,1)$ in~(i) by $[0,1)$ is of no consequence since if~(i) holds for $\alpha=0$ then it also holds for any $\alpha\in (0,1)$ and if $\alpha\in(0,1)$ then also $\alpha\in[0,1)$.

\begin{proof}
The equivalence of~(ii) and~(iii) is obvious, thus it remains to show the equivalence of~(i) and~(ii).
If $X\in\MM \alpha$ then we take $Z\in\MM 1$ such that $X\sim\alpha\circ Z$. Then $G_X(s)=G_Z(\alpha s)$ for $s\in[0,2]$. Since $G_Z(\alpha s)$ is completely monotone on $(0,1/\alpha)$ then $G_X(s)$ must be completely monotone on $(0,c)$ (hence on $(1,c)$) where $c=\min(2,1/\alpha)\in(1,2]$. For the converse, if $G_X(s)$ is completely monotone on $(1,c)$ (hence, on $(0,c)$) then we take $\alpha=1/c$ and apply part~(i) of Theorem~\ref{ThCM0}. 
\end{proof}

\begin{corollary}\label{v3} $X\sim N(Y)$ for some nonnegative $Y$ if and only if
there exists $f:[0,\infty)\to\mathbb{R}$ completely monotone on $(0,\infty)$ with $f(0)=f(0+)$ which agrees with $\GEN X(-s)$ on some $S\subset [0,1)$ containing an accumulation point. 

In this case there exists a nonnegative random variable $\tilde Y$ and $p_0\in(0,1]$ such that $\GEN X(-s)=p_0L_{\tilde Y}(s)$ for $s\in (0,1]$. If $p_0=1$ then $Y$ is a.s. zero and otherwise $Y$ satisfies  $\pp{Y\in dy}=p_0e^{y}\pp{\tilde Y\in dy}$.
\end{corollary}

We will demonstrate the use of Corollary~\ref{v3} in Example~\ref{Sibuya-1}.

\begin{proof}
If $X\sim N(Y)$ then $P_X(-s)=G_X(1+s)=L_Y(1+s)$ on $S=[0,1)$ where clearly $f(s)=L_Y(1+s)$ is completely monotone on $(0,\infty)$ with $f(0)=f(0+)$.

For the converse, if $P_X(-s)=f(s)$ on $S$, then $G_X(s)=f(s-1)$ on $S+1=\{s+1|\,s\in S\}\subset [1,2)\subset (0,2)$. Thus, if we define $g(s)=G_X(s)1_{[0,1)}(s)+f(s-1)1_{[1,\infty)}(s)$ then $g$ is completely monotone on $(0,1)\cup(1,\infty)$ and since $g$ is smooth on $(0,2)$ and in particular at $s=1$ it follows that $g$ is completely monotone on $(0,\infty)$ so that by~(ii) of Theorem~\ref{ThCM0} it follows that $X\sim N(Y)$ for some  nonnegative $Y$.

Now, clearly $p_0=P_X(0)=G_X(1)=L_Y(1)>0$. Hence $\pp{\tilde Y\in {\rm d}y}=p_0^{-1}e^{-y}\pp{Y\in {\rm d}y}$ defines a probability distribution satisfying $p_0L_{\tilde Y}(s)=L_Y(1+s)$. This implies that $P_X(-s)=p_0L_{\tilde Y}(s)$ and that $\pp{Y\in dy}=p_0e^y\pp{\tilde Y\in dy}$. Finally $p_0=1$ if and only if $L_Y(1)=1$ which is equivalent to $\pp{Y=0}=1$.
\end{proof}

\section{The sets \mbf{\MM \al} and \mbf{\MS \al}}\label{MMMS}

We recall that, for a given $\al\in[0,1]$, $\MM \alpha$ denotes the set of all $\alpha$-thinned distributions of $\MM 1$. Clearly, for each $\alpha\in[0,1]$, $\MM\alpha$ is nonempty since $\MM 1$ is nonempty.

For $\al\in[0,1)$ define $\MM {\al+}=\bigcap_{\ga\in(\al,1]}\MM{\ga}$. Note that $\MM 0$ contains only the distribution
of the constant zero.
Recall (Theorem~\ref{ThCM0}) that
$ X\in\MM{\al}$ if and only if $\GG{X}( s)=\GG{Z}(\alpha s)$ for some $Z\in\MM 1$ and all $s\in[0,2]$. Finally, denote $\MS 0=\MM{0+}$ and for $\al\in(0,1]$ let
\begin{align}\label{Malpha}
\MS\al\equiv\MM{\al}\setminus\bigcup_{\ga\in[0,\al)}\MM{\ga}\ .
\end{align}
By definition, $\{\MS\alpha|\,\alpha\in[0,1]\}$ are pairwise disjoint with $\bigcup_{\alpha\in[0,1]}\MS \alpha=\MM 1$.

\begin{theorem}\label{udi} For every $\al\in(0,1]$, $M_{\alpha+}=M_\alpha$. Moreover, $\MM {0+}\not=M_0$ and for $0<\al<\beta\le1$,
 \begin{equation}
 \MM 0\subset \MM{0+}\subset \MM\al\subset \MM{\beta}\,.
 \end{equation}
\end{theorem}

\begin{proof} $\MM {0+}$ contains any Poisson (as well as mixed Poisson) distribution, therefore it is non-empty and is different from $\MM 0$ which contains only the distribution of the constant $0$, which is also in $\MM{0+}$ and thus $\MM 0\subset \MM {0+}$. Let $X\sim \al\circ 1$ (a Bernoulli distribution with
$\PP{X=1}=\al$). Assume that for some $\beta\in (0,1]$ and some $Z\in\MM 1$ we have that $X\sim \beta\circ Z$. Then clearly also $Z$ has a Bernoulli distribution (possibly with $P(Z=1)=1$) and $\al=P(X=1)=\beta
P(Z=1)\le \beta$. Therefore, $X\not\in \MM \beta$ for any $\beta<\al$ and thus $X\in\MS \al$. Hence, $\MS\al$ is non-empty. We also see that if $X\sim\al\circ Z$, then for $\beta>\al$ we have that
$X\sim \beta\circ(\al/\beta)\circ Z$ and thus $\MM\al\subset \MM\beta$. The fact that they are contained in $\MM 1$ and contain $\MM 0$ is obvious. Finally, assume that for every $\beta>\al$, $X\sim \beta\circ
Z_\beta$ for some $Z_\beta\in\MM 1$. Then, since for $\al<\beta_1<\beta_2<1$ we have that $X\sim \beta_1\circ Z_{\beta_1}\sim \beta_2\circ(\beta_1/\beta_2)\circ Z_{\beta_1}$, we have that $Z_{\beta_2}\sim
(\beta_1/\beta_2)\circ Z_{\beta_1}$, where the right side is stochastically smaller than $Z_{\beta_1}$. Therefore $Z_\beta$ converges in distribution to some $Z_\al$ as $\beta\downarrow\al$ where $Z_\al$ may
possibly have a defective distribution (infinite with positive probability). Since $G_{Z_\beta}(s)=G_X(s/\beta)$ for $s\in [0,\beta]$ (hence for $s\in[0,\al]$), and $G_X(s/\beta)$ converges to $G_X(s/\al)$ for
every $s\in[0,\al]$, it follows that $G_{Z_\al}(s)=G_X(s/\al)$ for all $s\in[0,\al]$ and hence $G_X(s)=G_{Z_\al}(\al s)$ with $1=G_X(0)=G_{Z_\al}(0)$ which implies that the distribution of $Z_\al$ is proper and
hence $X\in \MM\al$. This implies that $\MM{\al+}=\MM\al$.
\end{proof}

For every  $X\in\MM{1}$ the value
$\nor X=\inf\{\al\in[0,1]:X\in\MM{\al}\}$ is well defined and
\begin{align}
X\in\MS{\al}\quad\Leftrightarrow\quad \nor X=\al.
\end{align}
The mapping $\nor{\cdot}:\MM{1}\to\sq R^+$ shares  some properties with a norm as follows.

\begin{theorem} \label{Th:properties} Let $X,Y\in\MM 1$ be independent and $\al\in[0,1]$. Then,
\begin{enumerate}
\item $\nor{{X+Y}}\leq \max(\nor{{X}},\nor{{Y}})\leq \nor{{X}}+\nor{{Y}}$,
\item $\nor{{\al\cc X}}=\al\nor{{X}}$,
\item When $X$ is a.s. bounded then $\rho(X)=0$ if and only if $X$ is a.s. $0$.
\end{enumerate}
\end{theorem}

\begin{proof}
(i) For every $\alpha,\beta\in[0,1]$ such that $X\in\MM \alpha\subset\MM {\max(\alpha,\beta)}$ and $Y\in\MM \beta\subset\MM{\max(\alpha,\beta)}$ we have $X+Y\in\MM {\max(\alpha,\beta)}$ and thus $\rho(X+Y)\le \max(\alpha,\beta)$. Setting $\alpha=\rho(X)$ when $\rho(X)>0$ or letting $\alpha\downarrow 0$ when $\rho(X)=0$ and similarly for $Y$ completes the proof.

(ii) Obviously $0\circ X\sim 0$ and thus $\rho(0\circ X)=0=0\cdot \rho(X)$, thus it suffices to prove this for $\alpha\in(0,1]$.

We first note that if $X\in \MM\beta$ then $\alpha\circ X\in\MM{\alpha\beta}$ and thus $\rho(\alpha\circ X)\le \alpha\beta$. By taking $\alpha=\rho(X)$ when $\rho(X)>0$ and letting $\beta\downarrow 0$ when $\rho(X)=0$, this implies that $\rho(\alpha\circ X)\le \alpha \rho(X)$. 

When $\rho(X)=0$ we have $0\le\rho(\alpha\circ X)\le \alpha\rho(X)=0$ and thus $\rho(\alpha\circ X)=\alpha\rho(X)=0$.

When $\rho(X)\in(0,1]$ then $X\sim \rho(X)\circ Z$ for some $Z\in\MM 1$. If there was a $\beta$ with $\nor{\al \cc
X}<\beta<\al\nor X$ then there was a $W\in\MM 1$ such that $\al\cc X\sim\beta \cc W$. Since $\beta/\al<\nor{X}\leq 1$ we can write
$\al\cc X=\al\cc (\afrac{\beta}{\al}) \cc W$ (recalling $\alpha>0$),
implying that $X\sim(\afrac{\beta}{\al}) \cc W$ and hence $\afrac{\beta}{\al}\geq \nor{X}$, contradicting $\beta<\alpha\rho(X)$.

(iii) Clearly, if $P(X=0)=1$ then $\rho(X)=0$. Conversly, when $\rho(X)=0$ then either $X\sim 0\circ Z\sim 0$ for some $Z\in\MM 1$ and we are done or for every $\alpha\in(0,1]$ there exists some $Z_\alpha\in \MM 1$ such that $X=\alpha\circ Z_\alpha$. Let $n$ be such that $P(X\le n)=1$. Then $P(Z_\alpha\le n)=1$ for all $\alpha\in(0,1]$. For every $1\le i\le n$ we have that $P(X=i)=\sum_{k=i}^n{k\choose i}\alpha^k(1-\alpha)^{i-k}P(Z_\alpha=k)$ which converges to zero as $\alpha\downarrow 0$ and thus $P(X=0)=1$.
\end{proof}

\begin{remark}
It is tempting to conjecture that  the left inequality in (i) of Theorem~\ref{Th:properties} is actually an equality. However, this turns out to be false in general. Another guess could be that it should always hold that $\rho(X+Y)\ge \min(\rho(X),\rho(Y))$. This is false as well. In Corollary~\ref{rho} to follow we will give an example where $\rho(X+Y)$ is strictly smaller than both $\rho(X)$ and $\rho(Y)$.
\end{remark}

\begin{remark}
It is interesting to identify a connection between Corollary~\ref{CorExt} and $\rho(X)$. First note that $\rho(X)=1$ if and only if $X\in\MS 1$ which holds if and only if  $G_X(s)$ is not completely monotone on $(1,c)$ for any $c\in (1,2]$. From the proof of Corollary~\ref{CorExt} it may be concluded that, when $\rho(X)\in[1/2,1)$, $X\in\MM {\rho(X)}$ if and only if $G_X(s)$ is completely monotone on $(0,1/\rho(X))$. In this case $X\sim \rho(X)\circ Z$ where $Z\in\MS 1$ so that $G_Z(s)$ is not completely monotone on $(1,c)$ for any $c\in(1,2]$ and thus $G_X(s)=G_Z(\rho(X)s)$ is not completely monotone on $(1/\rho(X),d)$ for any $d\in(1/\rho(X),2]$.

Assume that for some $n\ge 0$ we have that $\rho(X)\in[2^{-(n+1)},2^{-n})$. Let $Z\in \MS 1$ be such that $X\sim \rho(X)\circ Z\sim 2^{-n}\circ Z_n$ where $Z_n\sim (2^n\rho(X))\circ Z$. Then $G_{Z_n}(s)$ is completely monotone on $(0,1/(2^n\rho(X)))$ and is not completely monotone on $(1/(2^n\rho(X)),d]$ for any $d\in (1/(2^n\rho(X)),2]$.

This implies the following procedure. If $G_X(s)$ is completely monotone on $(1,2)$, then take $Z_1\in\MM 1$ such that $X\sim (1/2)\circ Z_1$. If $G_{Z_1}(s)$ is completely monotone on $(1,2)$ then take $Z_2\in\MM 1$ such that $Z_1\sim(1/2)\circ Z_2$. We continue like this until the first $n$ for which $G_{Z_n}(s)$ is not complely monotone on $(1,2)$. If it is not completely monotone on $(1,c)$ for any $c\in(1,2]$ then $\rho(X)=2^{-n}$. Otherwise we take the largest $c\in(1,2)$ for which $G_{Z_n}(s)$ is completely monotone on $(1,c)$ (necessarily strictly between $1$ and $2$) and then $\rho(X)=1/(2^nc)$. Note that if we take $Z\in \MS 1$ such that $Z_n\sim (1/c)\circ Z$ then we have that
\begin{equation}
X\sim (1/2)\circ Z_1\sim\ldots\sim(1/2^n)\circ Z_n\sim(1/2^n)\circ(1/c)\circ Z\sim\rho(X)\circ Z\ .
\end{equation}

How do we identify that $\rho(X)=0$? This will happen if for any $n$, $G_{Z_n}(s)$ is completely monotone on $(1,2)$, so that in this case this process never ends. Since $G_{Z_n}(s/2^n)$ is completely monotone on $(0,2^n)$ this will eventually result in a function which is completely monotone on $(0,\infty)$ and agrees with $G_X(s)$ on $[0,1)$ (hence, as discussed earlier, also on $[0,2]$), as expected (see Theorem~\ref{T1} below).
\end{remark}

If $0<\nor X\leq \al\leq 1$ then it makes sense to let $\aaa\cc X$ denote the random variable $Z$ on the right hand side of $X=\al\cc Z$, so that $X\sim\al\cc(\aaa\cc X)$ for every $\al\in[\nor X,1]$. On the other hand, from $\al\circ X\sim\al\circ Z$ for $Z\sim X$ it follows that $\al^{-1}\circ(\al\circ X)\sim X$ for all $\al\in(0,1]$.

When $\nor X=0$, then for $\alpha,s\in(0,1]$ we have
\begin{align*}
\LL{\al  (\al\inv\cc X)}(s)=\GG{(\al\inv\cc X)}(1-e^{-\al s})=\GG{X}\kla{\frac{1-e^{-\al s}}{\al s}\,s}
\end{align*}
which converges to $G_X(s)$ as $\alpha\downarrow 0$. Since $G_X(s)\to 1$ as $s\downarrow 0$, then by the continuity theorem for LST's (which also holds when restricting to $[0,1)$) we have that $\alpha(\alpha^{-1}\circ X)$ converges in distribution to some nonnegative $Y$ satisfying $G_X(s)=L_Y(s)$ on $[0,1)$ (hence, also on $[0,2]$). Therefore, $X\sim N(Y)$ as explained in the paragraph preceding Theorem~\ref{ThCM0}. Also, if $X\sim N(Y)$ then since $\alpha\circ N(Y/\alpha)\sim N(Y)$ for every $\alpha\in(0,1]$, this implies that $\rho(X)=0$. Therefore we have the following known result that can be traced back to \cite{mecke}, Satz 4.2 (see also \cite{gr}, Theorem 2.3 and \cite{jk}).

\begin{theorem}\label{T1}
$X\in\MS 0$ if and only if $X$ has a mixed Poisson distribution.
\end{theorem}

\begin{remark}
Since $N(Y)$ has unbounded support whenever $Y$ is not identically zero, any random variable with bounded support, other than the constant zero, cannot be in $\MS 0$.
\end{remark}

The following theorem generalizes Proposition 2.1 in \cite{jk}.

\begin{theorem}\label{PoiApp} [Poisson approximation] Let $S_n=\sum_{k=1}^n X_{kn}$, where $X_{1n},X_{2n},\ldots,X_{nn}$ are i.i.d. \intval random variables with $X_{1n}\dto X$ and $\ee X_{1n}\to EX$ as $n\to\infty$, where
$\ee{X}<\infty$. Also assume that $a_n\in[0,1]$ with $na_n\to \lambda\in(0,\infty)$ as $n\to\infty$. Then
\begin{align*}
a_n\cc S_n\dto N(\lambda\ee{X}),\qquad n\to \infty.
\end{align*}
\end{theorem}

\begin{proof} From \eqref{basic} it follows that for every $s\in[0,1]$,
\begin{align*}
\GG{a_n\cc S_n}(s)
=\GG{S_n}(a_ns)=\GG{X_{1n}}^n(a_ns)=\exp\left(\frac{\log G_{X_{1n}}(a_ns)}{a_ns}\cdot na_n s\right)\ .
\end{align*}
Thus, showing that the right hand side converges to $e^{-\lambda \ee{X} s}=L_{\lambda \ee X}(s)=G_{N(\lambda \ee{X})}(s)$, will complete the proof. Since $na_n\to\lambda$,
it remains to show that
\[\frac{\log G_{X_{1n}}(a_ns)}{a_ns}\to -\ee{X}\]
as $n\to\infty$. By Jensen's inequality it follows that
\begin{equation}
\frac{\log G_{X_{1n}}(a_n s)}{a_n s}=\frac{\log\EE{1-a_n s}^{X_{1n}}}{a_n s}\ge \ee X_{1n}\frac{\log(1-a_n s)}{a_n s}
\end{equation}
where the right hand side converges to $-\ee{X}$ as $n\to\infty$. Now, since $\log x\le x-1$ for $x>0$ we have that
\begin{equation}
\frac{\log G_{X_{1n}}(a_n s)}{a_n s}\le \frac{\ee{(1-a_n s)^{X_{1n}}}-1}{a_n s}\,.
\end{equation}

Note that for every $u\in(0,1)$ and every $x\in[0,\infty)$ we have (mean value theorem) that for some $v\in(0,u)$,
\begin{equation}
\frac{(1-u)^x-1}{u}=-x(1-v)^{x-1}\le -x(1-u)^{(x-1)^+}\,.
\end{equation}
Therefore, for every $\epsilon\in(0,1)$ and $n\ge s/\epsilon$ we have that
\begin{equation}
\frac{\ee{(1-a_n s)^{X_{1n}}}-1}{a_n s}\le -\ee{X_{1n}(1-a_n s)^{(X_{1n}-1)^+}}\le -\ee{X_{1n}(1-\epsilon)^{(X_{1n}-1)^+}}\,.
\end{equation}
As $x(1-\epsilon)^{(x-1)^+}$ is bounded and continuous (in $x$) on $[0,\infty)$, it follows that the right side converges, as $n\to\infty$, to $-\ee{X(1-\epsilon)^{(X-1)^+}}$, which in turn (monotone convergence)
converges to $-\ee{X}$ as $\epsilon\downarrow 0$. Therefore we have shown that
\begin{equation}
-\ee{X}\le \liminf \frac{\log G_{X_{1n}}(a_n s)}{a_n s}\le \limsup \frac{\log G_{X_{1n}}(a_n s)}{a_n s}\le -\ee{X}
\end{equation}
and we are done.
\end{proof}

\begin{remark}
As a sanity check, note that for $X_{1n}=1$ we have that $a_n\circ S_n\sim \Cbinom n{a_n}$ which implies the well known fact that $\Cbinom n{a_n}\dto\Cpoi{\lambda}$ whenever $na_n\to\lambda$. Also note that
$n^{-1}\circ S_n\to N(\ee X)$ which is a distributional version of a law of large numbers in the current context. Finally note that if either the supports of $X_{1n}$ are contained in a (common) bounded set or
$X_{1n}$ is stochastically increasing in $n$, then the condition $\ee{X_{1n}}\to \ee{X}$ is redundant.
\end{remark}

For $X\in \MM 1$, denote $\supp X=\{i|\, i\in\sqNN,\, \PP{X=i}>0\}$ (support of the distribution of $X$). We will say that $\supp X$ {\em contains holes} if for some $n\ge 1$, $P(X=n)>0=P(X=n-1)$. Note that a
special case is when $P(X=n)>0$ for some $n\ge1$ and $P(X=0)=0$. Also note that trivially the support of the constant zero does not contain holes, but that of any other positive integer
constant does.

\begin{lemma}\label{suppe} For every $\al\in(0,1)$, $\supp {\al\cc X}$ does not contain holes
and
\begin{equation}
\sup\{i|i\in\sqNN,\,P(X=n)>0\}=\sup\{i|i\in\sqNN,\,P(\al\circ X=n)>0\}\ .
\end{equation}
\end{lemma}

\begin{proof} Immediate from the fact that for each $n\ge 1$ with $P(X=n)>0$ we have, for every $0\le k\le n$, that
$P(\al\circ X=k)\ge {n\choose k}\al^k(1-\al)^{n-k}P(X=n)>0$.
\end{proof}

\begin{theorem}\label{thmstar}
$\nor X=\al\in(0,1)$ (equivalently, $X\in\MS\al$) if and only if there is some $Z\in\MS1$ such that $X\sim\al\circ Z$. Moreover, any $X\in\MM 1$ such that $\supp X$ contains holes is in $\MS 1$.
\end{theorem}

\begin{proof}
Assume that $\nor X=\al$ (equivalently, $X\in\MS\al$). Then there is a $Z\in\MM1$ such that $X\sim \al\circ Z$. Assume to the contrary that $Z\in \MM\beta$ for some $\beta\in(0,1)$. Then $Z\sim \beta\circ W$ for
some $W\in\MM1$ and thus $X\sim\al\circ(\beta\circ W)\sim (\al\beta)\circ W$ which contradicts the assumption that $\nor X=\al$. Thus, necessarily $Z\in\MS1$. Now, let $Z\in\MS1$ and take $X\sim \al\circ Z$.
Assume to the contrary that $X\in\MM\beta$ for some $\beta<\al$. Then, for some $W\in\MM 1$, $X\sim\beta\circ W$. Since $\al\circ (\beta/\al)\circ W\sim \beta\circ W\sim \al\circ Z$, then $Z\sim (\beta/\al)\circ
W$. This contradicts the assumption that $Z\in\MS1$.
Finally, by Lemma~\ref{suppe}, any distribution in $\MM1$ of which support contains holes cannot be in $\MM\al$ for any $\al\in(0,1)$ and is thus in $\MS1$.
\end{proof}

The next theorem shows that any mixed Poisson distribution not concentrated at zero is the distribution of a product of some random variable in $\MS\alpha$ (necessarily with unbounded support) with some independent Bernoulli random variable. 

\begin{theorem} Let $\al\in(0,1]$ and let $Y$ be a \nrv with $P(Y=0)<1$. Then there exist $X\in\MS\al$ and an independent Bernoulli $B$, such that
$BX\deq N(Y)$.
\end{theorem}
\begin{proof}
Let $P(Z=0)=0$ and for $n\ge 1$
\begin{equation}
P(Z=n)=\frac{1}{1-\LL Y(1/\al)}\ee{\left(e^{-Y/\alpha}\frac{(Y/\al)^n}{n!}\right)}\,.
\end{equation}
From Theorem~\ref{thmstar}, $Z\in\MS 1$ and thus $X\sim \al\circ Z\in\MS\al$. Let $B\sim\Cber{1-\LL Y(1/\al)}$ be independent of $Z$. It is easy to check that $BZ\sim N(Y/\al)$ and that
\begin{equation}
B(\al\circ Z)\sim \al\circ BZ\sim \al\circ N(Y/\al)\sim N(Y)
\end{equation}
and the proof is complete.\end{proof}

For a general $X\in\MM 1$ let $\zo X\sim X\given{(X\ge 1)}$. Obviously, as $\al\to 0$ the thinned variable $\al\cc X$ tends to zero, but $\zo{(\al\cc X)}$ may still converge weakly to a nonconstant random variable.

\begin{theorem}
Let $X\in\MM 1$ with $\PP{X=0}<1$. If $X$ has a finite mean then
\begin{align*}
&\zo{(\al\cc X)}\dto 1\as{\al\to 0}.
\intertext{If instead
$\PP{X\geq k}=k^{-\ga}L(k)$ with $\ga\in(0,1)$ and $L$  slowly varying at infinity then }
&\zo{(\al\cc X)}\dto \Csibuya{\ga}\as{\al\to 0}.
\end{align*}

\end{theorem}

\begin{remark} The Sibuya distribution is described in Example \ref{Sibuya}.
Recall that a function $L$ is slowly varying if $L(cx)/L(x)\to 1$ as $x\to\infty$  for every $c>0$. The theorem shows that if we apply binomial thinning repeatedly to an $X$ with finite mean then $\al^n\cc X$, given that it is still positive, will most probably be one if $n$ is large. This is no longer true in the infinite mean case. 
\end{remark}

\begin{proof}
The \apgf of $\zo X$ is given by
$\GG{\zo{X}}(s)=\frac{\GG{X}(s)-\GG{X}(1)}{1-\GG{X}(1)}$.
Provided that $\ee{X}<\infty$ we have
\begin{align*}
\GG{\zo{(\al\cc X)}}(s)=1-\frac{1-\GG{X}(\al s)}{1-\GG{X}(\al)}=1-\frac{\al s\ee{X}+o(\al)}{\al\ee{X}+o(\al)}\to 1-s,
\end{align*}
which is the \apgf of the constant one. Now assume that $a_k:=\PP{X\geq k}=k^{-\ga}L(k)$ and let $A(s)=\sum_{k=0}^{\infty} (1-s)^ka_k$. By Karamata's theorem for power series (\cite{BGT}, Corollary 1.7.3) $a_k=k^{-\ga}L(k)$ implies that
\begin{align*}
A(s)\sim s^{\ga-1} L(1/s)\Gamma(1-\gamma)
\end{align*}
as $s\to 0$. Now
\begin{align*}
sA(s)&=s\sum_{k=0}^{\infty} (1-s)^k\sum_{i=k}^\infty p_i
=s\sum_{i=0}^{\infty}p_i\sum_{k=0}^i  (1-s)^k
=1-(1-s)\GG{X}(s).
\end{align*}
{\em i.e.,} as $s\to 0$
\begin{align*}
1-\GG{X}(s)=1-\frac{1-sA(s)}{1-s}\sim \Gamma(1-\gamma) s^\gamma L(1/s).
\end{align*}
This implies that
\begin{align*}
1-\GG{\zo{(\al\cc X)}}(s)=\frac{1-\GG{X}(\al s)}{1-\GG{X}(\al)}\sim \frac{(\al s)^\gamma}{\al^\gamma}\frac{L(1/(\al s))}{L(1/\al)}\to s^\gamma
\end{align*}
as $\al \to 0$ by the properties of slowly varying functions. $G_{(\alpha\circ X)_0}(s)=1-s^\gamma$ is the \apgf of a Sibuya distribution.
\end{proof}

 Recall the following. For a sequence $\{a_k|\,k\ge 0\}$ define $r=\limsup_{k\to\infty}|a_k|^{1/k}$, then the radius of convergence of $S(x)=\sum_{k=0}^\infty a_kx^k$ is $R=0$ when $r=\infty$, $R=\infty$ when $r=0$ and $R=1/r$ when $r\in (0,\infty)$. When $|x|<R$, $S(x)$ converges absolutely and is infinitely differentiable with $S^{(n)}(x)=n!\sum_{k=n}^\infty {k\choose n}x^{k-n}a_k$ which also has a radius of convergence $R$ and thus converges absolutely for $|x|<R$.

\begin{theorem}\label{NN} Let $X\in\MM 1$ and let
$\rad X$ denote the radius of convergence  of $P_X(s)=G_X(1-s)=\sum_{k=0}^\infty s^kp_k$, where $p_k=\PP{X=k}$. Then $R_X\ge 1$ (possibly infinite) and the following holds:
\begin{enumerate}
\item If $(1+R_X)^{-1}<\al<1$ ($0$ on the left when $R_X=\infty$) then $\GG{X}(s/\al)$ is well defined for $s\in[0,1]$ and
\begin{align}\label{bj}
\ba j\equiv\sum_{k=j}^{\infty}{k\choose j}\al^{-k}(\al-1)^{k-j} p_k
\end{align}
converges absolutely with $\sum_{j=0}^\infty \ba j=1$. Moreover, if $X\in\MM\al$ with $X=\al\cc Z$, $Z\in\MM 1$ then necessarily $\ba j\ge 0$ for all $j\ge 0$ and
$\PP{Z=j}=\ba j$.
\item If $R_X>1$, $2(1+R_X)^{-1}<\alpha<1$ and $\ba j\geq 0$ for all $j\in\sqNN$ then $X\in\MM\al$ with $X=\al\cc Z$, $Z\in\MM 1$ and
$\PP{Z=j}=\ba j$.
\end{enumerate}
\end{theorem}

\begin{remark}\label{RNN}
Note that since $R_X\ge 1$, then $(1+R_X)^{-1}\le 1/2$ and if $R_X>1$ then $2(1+R_X)^{-1}<1$. In particular, the first part of Theorem~\ref{NN} is valid for any $\alpha\in(1/2,1)$. Also note that the first part also implies that when $(1+R_X)^{-1}<\alpha<1$ and there exists $j\ge 0$ such that $p_j^*(\alpha)<0$, then there is no $Z\in \MM 1$ such that $X\sim \alpha\circ Z$. In particular, if for any $\alpha\in (1-\epsilon,1)$, where $0<\epsilon\le 1/2$, there is some $j$ for which $p_j^*(\alpha)<0$ then necessarily $X\in \MS 1$. 
\end{remark}

\begin{proof}
\begin{enumerate}
\item $\GG{X}(s/\al)$ is well defined for $\al(1-\rad X)<s<\al(1+\rad X)$. Since $\al>(1+\rad X)^{-1}$ it follows that
$[0,1]\subset [\al(1-\rad X),\al(1+\rad X)]$. Since $(1-\alpha)/\alpha<R_X$, the series \eqref{bj} converges absolutely.  If $X=\al\cc Z$ then
\begin{align}
\GG{Z}(\al s)=\GG{X}(s),\quad |1-s|<\rad X.\label{f}
\end{align}
Since $\al>(1+\rad X)^{-1}$ it follows that \eqref{f} holds in particular for $s=1/\alpha$ and
\begin{align*}
\PP{Z=j}=\frac{(-1)^j}{j!}\GG{Z}\di j(1)&=\frac{(-1/\al)^j}{j!}\GG{X}\di j(1/\al)
=\ba j.\qedhere
\end{align*}
\item Since $\al>(1+\rad X)^{-1}$ the function
\begin{align}
f(s)&\equiv\GG X(s/\al)=\sum_{k=0}^\infty \kla{1-s/\al}^k p_k\nonumber\\
&=\sum_{k=0}^\infty\sum_{j=0}^{k} \left(\frac{1-\alpha}{\al}\right)^{k} p_k {k\choose j}\kla{\frac{1-s}{1-\al}}^j(-1)^{k-j}\label{sx}
\end{align}
is well defined for $s\in[0,1]$ and $f(0)=1$. The sum of the absolute values is
\begin{align*}
\sum_{k=0}^\infty \left(\frac{1-\alpha}{\al}\right)^{k} p_k \sum_{j=0}^{k}{k\choose j}\kla{\frac{1-s}{1-\al}}^j
&=\sum_{k=0}^\infty \left(\frac{1-\alpha}{\alpha}\right)^k\left(1+\frac{1-s}{1-\alpha}\right)^kp_k \\
&=\sum_{k=0}^\infty \left(\frac{2-\alpha-s}{\alpha}\right)^kp_k
\end{align*}
which converges provided that $|2-\alpha-s|/\alpha<R_X$. That is, when
\[-\alpha(R_X-1)-1<1-s<\alpha(R_X+1)-1\]
where the left hand side is negative and, since $\alpha\ge 2/(1+R_X)$, the right hand side is $\ge 1$.
Therefore for all $s\in[0,1]$ the sum $f(s)$ is absolutely convergent, so that we can interchange the order of summation in (\ref{sx}) and obtain, after some rearrangements,
\begin{align*}
f(s)=\GG X(s/\al)=\sum_{j=0}^\infty(1-s)^j\ba j.
\end{align*}
If $\ba j\geq 0$ for $j\in\sqNN$ then, since $\sum_{j=0}^\infty p_j^*(\alpha)=1$, $f(s)$ is the \apgf of a random variable $Z\in\MM 1$ with $\PP{Z=j}=\ba j$ and
 $X=\al\cc Z$.
\end{enumerate}
\end{proof}

We have seen that if the support of a distribution in $\MM 1$ contains a hole then $X\in\MS 1$. In Theorem~\ref{thmstar} in the next section we will show that for distributions with bounded support the converse is also true. The following example demonstrates that this converse fails when the support is unbounded. That is, there are distributions in $\MS 1$ with unbounded supports that do not contain holes.

\begin{example}

For $0<d<c<1$ let $p_k=P(X=k)$ be defined as follows.
\begin{align*}
p_k=\frac{1}{A}\cdot \begin{cases}
c^k&k\text{ even},\\
d^k&k\text{ odd}
\end{cases}
\end{align*}
where 
$A=(1-c^2)^{-1}+d(1-d^2)^{-1}$.
The support of this distribution is unbounded and does not contain holes. Clearly,
\begin{align}
2AP_X(s)&=2\left(\frac{1}{1-(cs)^2}+\frac{ds}{1-(ds)^2}\right)=\frac{1}{1-cs}+\frac{1}{1+cs}+\frac{1}{1-ds}-\frac{1}{1+ds}
\end{align}
so that for odd $n$ we have
\begin{equation}
\frac{2A}{n!}P_X^{(n)}(s)=\frac{c^{-1}}{(c^{-1}-s)^{n+1}}-\frac{c^{-1}}{(c^{-1}+s)^{n+1}}+\frac{d^{-1}}{(d^{-1}-s)^{n+1}}+\frac{d^{-1}}{(d^{-1}+s)^{n+1}}\ .
\end{equation}
Since $d<c$ then for $s\in(-1,0)$ we have $c^{-1}+s<\max(c^{-1}-s,d^{-1}-s,d^{-1}+s)$ from which it follows that
\begin{equation}
(c^{-1}+s)^{n+1}\cdot \frac{2A}{n!}p_X^{(n)}(s)\to -c^{-1}<0
\end{equation}
when $n$ is odd and $n\to\infty$. Therefore $P_X(s)$ is not absolutely monotone on $(-\epsilon,0)$ for any $\epsilon\in(0,1)$. From Corollary~\ref{CorExt}, this is equivalent to $X\in\MS 1$.

\end{example}

\section{Distributions with bounded support}\label{DBS}

We now consider random variables $X$ with support in the finite set $\II n\equiv\{0,1,2,\ldots,n\}$. Let us write
$\MMF \al=\{X\in\MM\al: \supp{X}\subseteq \II n\}$ and $\MSF \al=\{X\in\MS\al: \supp{X}\subseteq \II n\}$ for the respective subsets of $\MM\al$ and $\MS\al$. Let as before for $\al\in(0,1]$
\begin{align}\label{bj2}
\ba j=\sum_{k=j}^{n}{k\choose j}\al^{-k}
(\al-1)^{k-j} p_k,
\end{align}
where $p_k=\PP{X=k}$. We note that here $\GG{X}(s)$, being a polynomial, is well defined and finite for all $s\in\mathbb{R}$.

\begin{theorem}\label{UH2} For $n\ge 1$ let $X\in\MMF 1$ with $p_n>0$. Then $\rho(X)\ge p_n^{1/n}>0$ and for $\alpha\in(0,1]$, $X\in\MMF \alpha$ if and only if $p_j^\star(\alpha)\ge 0$ for all $j\in \II {n-1}$, in which case $X\sim \alpha\circ Z$ with $P(Z=j)=p_j^\star(\alpha)$ for $j\in\II n$.
\end{theorem}

\begin{proof} First observe that for every $\alpha\in(0,1]$ we have that $\ba n=p_n/\alpha^n$ and that $p_n^*(\rho(X))\le 1$. This implies that $\ba n> 0$ for all $\alpha\in(0,1]$ and that $\rho(X)\ge p_n^{1/n}$.  Now, for every $\alpha\in(0,1]$ and $s\in\mathbb{R}$, we have,
\begin{align}
\GG{X}(s/\al)&=\sum_{k=0}^n (1-s/\al)^kp_k=\sum_{k=0}^n (1-s)^kp_k\sum_{j=0}^{k}{k\choose j}\al^{-k}
(1-s)^{j-k}(\al-1)^{k-j}\nonumber
\\&=\sum_{j=0}^n(1-s)^j \sum_{k=j}^{n}{k\choose j}\al^{-k}
(\al-1)^{k-j} p_k=
\sum_{j=0}^n (1-s)^j\ba j\,,
\end{align}
which, setting $s=0$, gives $\sum_{j=0}^np_j^\star(\alpha)=\GG{X}(0)=1$. Therefore, $\GG{X}(s/\alpha)$ is the \apgf of some $Z\in\MMF 1$ ({\em i.e.,} $X\sim\al\cc Z$, necessarily with $P(Z=j)=p_j^\star(\alpha)$ for $j\in\II n$). 

\end{proof}

\begin{remark}\label{rem5}
We mention that since $\MMF {\alpha_1}\subset \MMF {\alpha_2}$ for $0<\alpha_1<\alpha_2\le 1$, then it follows that $\ba i\ge 0$ for all $i\in\II n$ and $\alpha\in[\rho(X),1]$. For all $\alpha\in(0,\rho(X))$ there is at least one $i\in\II {n-1}$ for which $\ba i<0$. This will be illustrated in Remark~\ref{rem8} following Example~\ref{ex02}.
\end{remark}

\begin{theorem} \label{hole} Let $X\in\MM 1$ with bounded support. Then $X\in\MS 1$ if and only if its support contains a hole.
\end{theorem}

\begin{proof}
We have already seen (Lemma~\ref{suppe}) that if the support of $X$ has a hole then it must be in $\MS 1$. For the converse, assume that the support of $X$ has no holes. Therefore, either $P(X=0)=1$ and then $X\in\MS 0$ (hence, not in $\MS 1$), or we let $n=\max\{i|\,P(X=i)>0\}$. Since $G_X(s)$ is a polynomial of degree $n$, then for all $k>n$ we have that $G^{(k)}(s)=0$ for all real $s$. For $k\le n$ we have that $(-1)^kG^{(k)}(1)=P(X=k)>0$ so that by the continuity of $G^{(k)}(s)$ there must be some $c_k>0$ such that $(-1)^kG^{(k)}(s)>0$ for all $s\in(1,c_k)$. Taking $c= \min\{c_k|\,k=0,\ldots,n\}>1$ gives that $G_X$ is completely monotone on $(0,c)$ and by Corollary~\ref{CorExt} we have that $X\not\in\MS 1$.
\end{proof}

The following implies that $\rho(\cdot)$ is continuous on $\MMF 1\setminus \DSYM_1(n-1)$ (but not on $\MMF 1$). Although it is possible that this may be concluded from the continuity property of (possibly complex) roots of polymials of the form $x^n+\sum_{i=0}^{n-1}c_ix^i$ in $c_0,\ldots,c_{n-1}$ (and the continuity of the maximum function), we prefer to give a simple self contained derivation demonstrating how it is in fact a by product of Theorem~\ref{hole}.

\begin{theorem}\label{limrho}
Assume that $X_k\in \MMF 1$ for $k\ge 1$ and that $X_k\dto X$ with $P(X=n)>0$. Then $\rho(X_k)\to \rho(X)>0$ as $k\to\infty$.
\end{theorem}

\begin{proof}
First assume that $\rho(X_k)=1$ for all $k\ge 1$. Since $P(X_k=n)\to P(X=n)>0$ as $k\to\infty$ then there exists some $K$ such that $P(X_k=n)>0$ for all $k\ge K$. By Theorem~\ref{hole} for each $k\ge K$ the support of $X_k$ contains a hole (which cannot be $n$) and thus there exists some $i\in\{0,\ldots,n-1\}$ and a subsequence $X_{k_m}$ such that $P(X_{k_m}=i)=0$. Therefore $P(X=i)=0$ and since $P(X=n)>0$ then the support of $X$ contains a hole and thus $\rho(X)=1$.

In general, let $\alpha_k=\rho(X_k)$ and let $Z_k\in\MSF 1$ be such that $X_k\sim\alpha_k\circ Z_k$. It suffices to show that if $\alpha_k$ converges then it necessarily converges to $\rho(X)$. The reason is that this would imply that any convergent subsequence of $\alpha_k$ necessarily converges to $\rho(X)$ which in turn implies that $\alpha_k\to \rho(X)$. Therefore, we assume that $\alpha_k\to\alpha$ as $k\to\infty$. Take $Z_k\in\MSF 1$ such that $X_k\sim\alpha_k\circ Z_k$. There exists a subsequence $Z_{k_m}$ that converges in distribution to some $Z$. Now,
\begin{equation}
\alpha^n P(Z=n)=\lim_{m\to\infty}\alpha_{k_m}^nP(Z_{k_m}=n)=\lim_{m\to\infty}P(X_{k_m}=n)=P(X=n)>0
\end{equation}
and thus $\alpha>0$ and $P(Z=n)>0$. By the first part of the proof we have that $Z\in \MSF 1$ and since $X_{k_m}\sim \alpha_{k_m}\circ Z_{k_m}\dto  \alpha\circ Z$ we necessarily have that $X\sim \alpha\circ Z$ and thus $\rho(X)=\alpha>0$.
\end{proof}

\begin{remark}
We note that without the condition $P(X=n)>0$, Theorem~\ref{limrho} would no longer be always valid. For example, for $n\ge 2$ take $p_0,\ldots,p_{n-2}>0$ with $\sum_{i=0}^{n-2}p_i=1$ and assume that $P(X_k=i)=(1-k^{-1})p_k$ for $0\le i\le n-2$, $P(X_k=n-1)=0$ and $P(X_k=n)=k^{-1}$. For each $k\ge 1$, $\rho(X_k)=1$ since the support of $X_k$ contains a hole. However, as $k\to\infty$ we have that $X_k\dto X$ where the support of $X$ is $\II {n-2}$ and contains no holes and thus $\rho(X)<1$. An extreme case is when $n=2$ and then $\rho(X_k)=1$ for $k\ge 1$ but $\rho(X)=0$.
\end{remark}

\begin{corollary}\label{rho}
There exist independent $X,Y\in M_1$ such that $\rho(X+Y)$ is strictly smaller than both $\rho(X)$ and $\rho(Y)$.
\end{corollary}

\begin{proof}
Suppose that $X,Y$ are independent and both have support $S=\{0,1,3,4\}$. Since $2\not\in S$ then $S$ contains a hole and thus $\rho(X)=\rho(Y)=1$. However the support of $X+Y$ is $S+S=\{i|\,0\le i\le 8\}$ which has no holes and thus $\rho(X+Y)<1$.
\end{proof}

\begin{remark}
If one is not happy with the example given in the proof of Corollary~\ref{rho} and is looking for an example with $\rho(X)\not=\rho(Y)$, then one may replace $X$ by $\alpha\circ X$. Since
$\alpha\circ X+Y\dto {X+Y}$ as $\alpha\uparrow 1$ and $P(X+Y=8)>0$, then Theorem~\ref{limrho} implies that $\rho(\alpha\circ X+Y)\to \rho(X+Y)<1$ as $\alpha\uparrow 1$. Therefore there exists an $\alpha\in (0,1)$ such that
\begin{equation}\label{eq:circXY}
\rho(\alpha\circ X+Y)<\alpha=\rho(\alpha\circ X)<1=\rho(Y)\ .
\end{equation}
For an example with 
\begin{equation}
\rho(X+Y)<\min(\rho(X),\rho(Y))<\max(\rho(X),\rho(Y))<1\,,
\end{equation} 
simply replace $X$ and $Y$ by $(\alpha\beta)\circ X$ and $\beta\circ Y$ where $\beta\in(0,1)$ and $\alpha$ is what we chose for (\ref{eq:circXY}).
\end{remark}

\section{Examples}\label{Examples}

\begin{example}\label{example1} We have already seen that a Bernoulli random variable $B$ with $\PP{B=1}=\alpha$ is a member of $\MS{\alpha}$. Since the constant $n$ is in $\MS 1$ a binomial random variable $\alpha\circ
n\sim\Cbinom n\al$ is also a member of $\MS \alpha$.
\end{example}

\begin{example}\label{ex02}
Assume that the support of $X$ is $\{0,1,2\}$ and that $p_i=P(X=i)>0$ for $i=0,1,2$. Then it is easy to check that
\begin{align}
\ba 0&=p_0-\left(\frac{1}{\alpha}-1\right)p_1+\left(\frac{1}{\alpha}-1\right)^2p_2\nonumber\\
\ba 1&=\frac{2p_2}{\alpha}\left(1+\frac{p_1}{2p_2}-\frac{1}{\alpha}\right) \\
\ba 2&=\frac{p_2}{\alpha^2}\ge 0\quad \forall\alpha\in(0,1] \,.\nonumber
\end{align}
The quadratic function $P_X(-x)=p_2x^2-p_1x+p_0$ is nonnegative when $\Delta=p_1^2-4p_0p_2\le 0$ which implies that $p_0^*(\alpha)\ge 0$ for $\alpha\in(0,1]$. Therefore, in this case the equation for $p_1^*(\alpha)$ implies that $\rho(X)=\frac{1}{1+\frac{p_1}{2p_2}}$.

When $\Delta>0$ the two (positive) roots of $P_X(-x)$ are $\frac{p_1\pm\sqrt{\Delta}}{2p_2}$. From this it follows that $p_0^*(\alpha)\ge 0$ when either
\begin{equation}\label{eq:ineq1}\frac{1}{1+\frac{p_1-\sqrt{\Delta}}{2p_2}}\le\alpha\le 1\end{equation} 
or when 
\begin{equation}\label{eq:ineq2}0<\alpha\le \frac{1}{1+\frac{p_1+\sqrt{\Delta}}{2p_2}}\,.\end{equation}
Evidently, (\ref{eq:ineq2}) is irrelevant since the right hand side is strictly less than $(1+\frac{p_1}{2p_2})^{-1}$ and for any such $\alpha$, $p_1^*(\alpha)<0$. Any $\alpha$ that satisfies (\ref{eq:ineq1}) is also larger than $(1+\frac{p_1}{2p_2})^{-1}$ and for such $\alpha$ we also have that $p_1^*(\alpha)\ge 0$. Summarizing, with $\Delta^+=\max(\Delta,0)$,  we have that
\begin{equation}\label{eq:rho3}
\rho(X)=\frac{1}{1+\frac{p_1-\sqrt{\Delta^+}}{2p_2}}=\left\{\begin{array}{ll}\frac{1}{1+\frac{p_1}{2p_2}}&\Delta< 0\\ \\
\frac{1}{1+\frac{p_1}{2p_2}}=\frac{1}{1+\frac{2p_0}{p_1}}&\Delta=0\\ \\
\frac{1}{1+\frac{2p_0}{p_1+\sqrt{\Delta}}} &\Delta > 0\,.
\end{array}\right.
\end{equation}

Therefore, if $X\sim\rho(X)\circ Z$, noting that $p_1^*(\rho(X))=0$ when $\Delta\le 0$ and that $p_0^*(\rho(X))=0$ when $\Delta\ge 0$ (recalling that $p_2^*(\alpha)>0$ for all $\alpha\in(0,1]$ and in particular for $\alpha=\rho(X)$), we have
\begin{equation}
\supp{Z}=\left\{\begin{array}{ll}
\{2\}& \Delta=0\\
\{1,2\}&\Delta>0\\
\{0,2\}&\Delta<0\,.
\end{array}\right.
\end{equation}
Not surprisingly, these are also all of the possible supports which contain at least one hole.

When $\Delta=0$, since $p_0^*(\rho(X))=p_1^*(\rho(X))=0$, we necessarily have that $p_2^*(\rho(X))=1$. Therefore,
\begin{itemize}
\item When $\Delta=0$, $X\sim\Cbinom{2}{\rho(X)}$ (with $P(Z=2)=1$).
\item When $\Delta>0$, $Z-1\sim\Cber{p_2^*(\rho(X))}$.
\item When $\Delta<0$, $Z/2\sim\Cber{p_2^*(\rho(X))}$.
\end{itemize}

When $p_i=1/3$ for $i=0,1,2$ we have that $\Delta=1/3-4(1/3)^2<0$ and thus $\Delta^+=0$ so that
\begin{equation}
\rho(X)=\frac{1}{1+\frac{p_1}{2p_2}}=\frac{1}{1+\frac{1/3}{2/3}}=\frac{1}{1+\frac{1}{2}}=\frac{2}{3}
\end{equation}
which is consistent with Example~\ref{DU} appearing later. In this case $Z/2\sim\Cber{3/4}$ since
\begin{equation}p_2^*(\rho(X))=\frac{p_2}{\rho(X)^2}=\frac{1/3}{(2/3)^2}=\frac{3}{4}
\end{equation}
Equivalently, $\frac{2}{3}\circ [2(\frac{3}{4}\circ 1)]$ has the uniform distribution on $\{0,1,2\}$.
\end{example}

\begin{remark}\label{rem8}
To illustrate Remark~\ref{rem5} assume that $\Delta>0$, set 
\begin{equation}
\alpha_1=\left(1+\frac{p_1+\sqrt{\Delta}}{2p_2}\right)^{-1}\ 
\alpha_2=\left(1+\frac{p_1}{2p_2}\right)^{-1}\ \alpha_3=\rho(X)=\left(1+\frac{p_1-\sqrt{\Delta}}{2p_2}\right)^{-1}\,.
\end{equation}
and recall that $\ba 2> 0$ for all $\alpha\in(0,1]$. Now,
\begin{enumerate}
\item For $\alpha\in [\alpha_3,1]$, $\ba 0,\ba 1\ge 0$.
\item  For $\alpha\in [\alpha_2,\alpha_3)$, $\ba 0<0\le \ba 1$.
\item For $\alpha\in(\alpha_1,\alpha_2)$, $\ba 0,\ba 1<0$.
\item For $\alpha\in(0,\alpha_1]$, $\ba 1<0\le \ba 0$.
\end{enumerate}
As observed in Remark~\ref{rem5}, we have that for $\alpha\in[\rho(X),1]$, $\ba i\ge 0$ for $i\in\II 2$ and for $\alpha\in(0,\rho(X))$ there is at least one $i\in \II 1$ such that $\ba i<0$.
\end{remark}

\begin{remark}
Denoting the right hand side of (\ref{eq:rho3}) by $r(p_0,p_1,p_2)$ and defining $r(1,0,0)=0$ it is easy to check that $r(\cdot)$, defined on the closed simplex 
\begin{equation}
S=\{(p_0,p_1,p_2)|\,p_0,p_1,p_2\ge 0,\,p_0+p_1+p_2=1\}\ ,
\end{equation} 
gives the correct formula for $\rho(X)$ also for the cases where at least one of the $p_i$'s is zero. In particular, when $p_2=0$ and $0<p_1<1$, $X\sim\Cber{p_1}$ with $\rho(X)=p_1$ (see Example~\ref{example1}). For $(0,p_1,p_2)$ with any choice of $p_1,p_2$ or $(p_0,0,p_2)$ with $p_2>0$ the support contains a hole so that $\rho(X)=1$. For $(1,0,0)$, $\rho(X)$ is clearly zero. Also, as mentioned in Theorem~\ref{limrho}, $r$ is continuous on $\{p|\,p\in S,p_2>0\}$ as well as on $\{p|\,p\in S,p_1>0,p_2=0\}$.
\end{remark}

\begin{example} Poisson and more generally mixed Poisson random variables are in $\MS 0$ (see Theorem~\ref{T1}).
If $X$ has a Poisson distribution and $X=N(Y)$ then $Y$ is a.s. a constant.
\end{example}

\begin{example} \label{negb} $X$ has a generalized negative binomial distribution with probability function (see, {\em e.g.,} \cite{gj,kotz}) if, for some $a>0$,
\begin{align}\label{negbin}
\PP{X=k}=\frac{\Gamma(a+k)}{\Gamma(a)k!}(1-p)^kp^a\,,
\end{align}
for $k\ge 0$. In this case 
\begin{equation}
G_X(s)=\left(\frac{\frac{p}{1-p}}{\frac{p}{1-p}+s}\right)^\alpha\,.
\end{equation}
for $s\in\big(-\frac{p}{1-p},2+\frac{p}{1-p}\big)\supset [0,2]$ and does not exist elsewhere. As a function on $[0,\infty)$ the right side is clearly completely monotone on $(0,\infty)$ and continuous at zero. It is in fact the LST of $Y\sim\Cgam{\alpha}{\frac{p}{1-p}}$. Hence ((ii) of Theorem~\ref{ThCM0}) the generalized negative binomial distribution is in $\MS 0$. Of course, the geometric and the negative binomial distributions, where in both we count only failures, are special cases.
\end{example}

\begin{example}\label{Liu}[Conditional geometric] If $X\sim\text{Uniform}(\II {m-1})$ then 
\begin{align*}
\GEN X(s)=\frac{1}{m}\frac{1-s^{m}}{1-s}
\end{align*}
and, \changed{as will be shown in Example \ref{DU},} $\rho(X)=1-\frac{1}{m}$. We know that the generating function $\GEN X(s)$ is absolutely monotone on $(z,0)$ with $z=1-\frac{1}{\rho(X)}=-\frac{1}{m-1}$, but not on $(z-\veps,0)$ for any $\veps>0$. Now suppose that $\gamma\in(0,1]$ is a given probability. It follows that the function
\begin{align*}
g(s)=\frac{1-(\gamma s)^m}{1-\gamma s}
\end{align*}
is absolutely monotone on $(z/\gamma,0)$ and not on $(z/\gamma-\veps,0)$. In other words $g$ is \changed{a constant multiple} of the \pgf of some random variable $Y$ with $\rho(Y)=\frac{1}{1-z/\gamma}=\frac{(m-1)\gamma}{(m-1)\gamma+1}$. In fact  
\begin{align*}
\GEN Y(s)=\frac{1-\gamma}{1-\gamma^m}\frac{1-(\gamma s)^m}{1-\gamma s}=\frac{1-\gamma}{1-\gamma^m}\sum_{i=0}^{m-1} \gamma^is^i
\end{align*}
so $Y$ is a geometric r.v. conditioned to be smaller than $m$. This distribution and its convolution powers were studied in \cite{what}, where one has to set  $\gamma=\alpha/\beta$ and  $\alpha<\beta$.

The above distribution also makes sense if $\gamma>1$, {\em i.e.,}  $\delta:=1/\gamma$ is a probability. Then
\begin{align*}
\GEN Y(s)=\delta^{m-1}\frac{1-\delta}{1-\delta^{m}}\sum_{i=0}^{m-1} \delta^{i} s^{m-1-i},
\end{align*}
so $Y$ has the distribution of $m-1-X$, where $X$ is geometric, given $m-1-X\geq 0$. This is the distribution in \cite{what} with $\alpha>\beta$.
\end{example}

\begin{example}\label{Sibuya}
In a sequence of
independent experiments with decreasing success probabilities $\afrac{\ga}{k}$, $k=1,2,\ldots$, $\ga\in(0,1)$, the number of trials $X$ until the first success is observed has a so called \lex{Sibuya distribution}
(\cite{dev1}), with probabilities given by
\begin{align*}
\PP{X=k}=\frac{\ga}{k}\prod_{j=1}^{k-1}\kla{1-\frac{\ga}{j}}=(-1)^{k+1}{\ga\choose k},\qquad k=1,2,3,\ldots.
\end{align*}
We write $X\deq\Csibuya\ga$ in this case. The associated \apgf is
$\GG{X}(s)=1-s^\ga$. It follows from Hardy and Littlewood's Tauberian theorem that
$\PP{X\geq k}\sim \afrac{k^{-\ga}}{\Gamma(1-\ga)}$ as $k\to\infty$, in particular $\ee{(X^c)}=\infty$ for $c\geq \ga$, implying that $X$ has an infinite mean. Also, since the support of $X$ has a gap at zero,
$X\in\MS 1$.

If $X\deq\Csibuya{\ga}$ then $\al\cc X$ has a
 \lex{scaled Sibuya distribution} (\cite{cs}) with scale parameter $\al$, {\em i.e.,}
$\GG{\al\cc X}(s)=1-(\al s)^\ga$ and
\begin{align*}
\PP{\al\cc X=k}=\begin{cases}
1-\al^\ga&k=0\\
(-1)^{k+1}\al^\ga{\ga\choose k}&k=1,2,3,\ldots.
\end{cases}
\end{align*}
Obviously, $\alpha\circ X\in\MS\al$.
\end{example}

\begin{example}\label{Sibuya-1}
With the notations from Example~\ref{Sibuya}, consider $X-1$, the number of failures until the first success, when $X\sim\Csibuya\ga$. It is known ({\em e.g.,} \cite{dev1}) that $X-1\sim N(Y)$ where $Y\sim \frac{Z_1Z_{1-\gamma}}{Z_\gamma}$ and $Z_1,Z_{1-\gamma},Z_\gamma$ and $N(\cdot)$ are independent with $Z_\alpha\sim\Cgam{\alpha}{1}$ for $\alpha=1,1-\gamma,\gamma$ and hence $X-1\in \MS 0$.

We would like to demonstrate the use of Corollary~\ref{v3} to infer that $X-1\in\MS 0$, thereby giving an alternate (necessarily equivalent) representation of $Y$.

If $U\sim\Cunif{0}{1}$ and $V\sim\Cgam{1-\gamma}{1}$ are independent then, for $s> 0$,
\begin{equation}
L_{UV}(s)=\ee L_V(Us)=\ee\left(\frac{1}{1+sU}\right)^{1-\gamma}=\int_0^1(1+su)^{\gamma-1}{\rm d}u=\frac{(1+s)^\gamma-1}{\gamma s}\,,
\end{equation}
so that, with $p_0=\pp{X-1=0}=\pp{X=1}=\gamma$ and $\tilde Y=UV$,
\begin{equation}
\GEN {X-1}(-s)=\frac{(1+s)^\gamma-1}{s}=p_0\ee e^{-s\tilde Y}
\end{equation}
for every $s\in (0,1]$, where the right side is well defined on $[0,\infty)$ and completely monotone on $(0,\infty)$. Thus, if we take $Y$ with $\pp{Y\in dy}=\gamma e^y\pp{\tilde Y\in dy}$, then by Corollary~\ref{v3}, $X-1\sim N(Y)$ where $Y$ and $N$ are independent.

This necessarily implies that 
\begin{equation}
\frac{L_{UV}(s-1)}{L_{UV}(-1)}=L_{\frac{Z_1Z_{1-\gamma}}{Z_\gamma}}(s)\,,
\end{equation}
where we observe that $UV\sim e^{-Z_1}Z_{1-\gamma}$.

Although this is not the main point of this example, for the sake of completeness, it is straightforward to show that the density of $\tilde Y$ is given by
\begin{equation}
\frac{\Gamma(-\gamma,y)}{\Gamma(1-\gamma)}
\end{equation}
where $\Gamma(-\gamma,y)=\int_y^\infty e^{-t}t^{-\gamma-1}dt$ is the upper incomplete Gamma function. Since $\Gamma(-\gamma,y)=e^{-y}U(1+\gamma,1+\gamma,y)$ where $U$ is the confluent hypergeometric function of the second kind,  the density of $Y$ can be written as 
\begin{equation}\frac{\gamma U(1+\gamma,1+\gamma,y)}{\Gamma(1-\gamma)}=\frac{U(1+\gamma,1+\gamma,y)}{-\Gamma(-\gamma)}\,.
\end{equation}
A related expression may be found in Proposition~6 of~\cite{KP17}.\end{example}

\begin{example}[Mixtures]
Let $(\Theta,\mathcal{G},Q)$ be some probability space and assume that, for each $\theta\in\Theta$, $p(\theta,\cdot)$ is a probability mass function on $\mathbb{Z}_+$ and that for each $n\in \mathbb{Z}_+$, $p(\cdot,n)$ is $\mathcal{G}$-measurable. That is, $P(\theta,A)=\sum_{n\in A}p(\theta,n)$, where $\theta\in\Theta$ and $A\subset\mathbb{Z}_+$, is a (normal) Markov kernel. Let $K_\theta$ be such that $P(K_\theta=n)=p(\theta,n)$, let $K_\Theta$ be such that $P(K_\Theta=n)=\int_{\Theta}p(\theta,n)Q(\text{d}\theta)$ and 
for $\alpha\in (0,1)$ let $(\alpha\circ K)_\Theta$ satisfy

\begin{equation}\label{eq:aKt}
P((\alpha\circ K)_\Theta=n)=\int_\Theta P(\alpha\circ K_\theta=n)Q(\text{d}\theta)\,.
\end{equation}
Then it is straightforward to check that $(\alpha\circ K)_\Theta\sim \alpha\circ K_\Theta$. Similarly if  $\al>\kappa\equiv \sup_{\theta\in\Theta}\nor{K(\theta)}$ then
\begin{align*}
\al\inv\cc \gv K\Theta\deq \gv{(\al\inv\cc K)}\Theta,
\end{align*}
where $(\alpha^{-1}\circ K)_\Theta$ obeys (\ref{eq:aKt}) with $\alpha^{-1}$ replacing $\alpha$,
so that $\gv K\Theta\in\MM\al$ and $\nor{\gv K\Theta}\geq\kappa$. 

The mixed Poisson distribution is, of course, a special case of this setup. For this case $\alpha\circ K_\Theta$ takes the form $\alpha\circ N(Y)$, while for $(\alpha\circ K)_\Theta$ we first take $\alpha\circ N(y)\sim N(\alpha y)$ and then uncondition with respect to $y$ to obtain $N(\alpha Y)$. This results in $\alpha\circ N(Y)\sim N(\alpha Y)$ which we have already seen earlier. For this case for all $y\ge 0$ we have that $\rho(N(y))=\kappa=\rho(N(Y))=0$.
\end{example}

\section{Binomial mixtures}\label{BM}
There is a second mixing operation that can be performed on binomial random variables.

We first note that if $X,U_1,U_2,\ldots,$ are independent random variables with $X\in\MM 1$ and $U_i\sim\Cunif{0}{1}$, then, for $\alpha\in[0,1]$ we have that
\[\alpha\circ X\sim \sum_{i=1}^X 1_{\{U_i\le \alpha\}}\]
Let us replace $X$ by $n\in\mathbb{Z}_+$ and $\alpha$ by some independent $W$ with support in $[0,1]$. Denote a random variable having such a distribution by $W\circ n$, that is, 
\[W\circ n\sim \sum_{i=1}^n 1_{\{U_i\le W\}}\]
and let us call this distribution a {\em binomial mixture}. If we take a Poisson process $N$ with rate $1$, independent of $W$, then this is also the conditional distribution of $N(W)$ given $N(1)=n$.

Loosely speaking, this means that the conditional distribution of $X$ given $W=p$ is $\Cbinom{n}{p}$.
Binomial mixtures appear for example in connection with exchangeable experiments: If $Y_1,Y_2,Y_3\ldots$ is an infinite sequence of exchangeable Bernoulli random variables then de Finetti's theorem
(\cite{feller},VII.4) implies that $S_n=Y_1+Y_2+\ldots+Y_n$ is a binomial mixture. In fact, for this case the same $W$ can be used for every $n$. It is to be noted that the requirement that the sequence $\seq{Y}{n}{\sqNN}$ is infinite
is crucial in this setup.

The \apgf of $W\cc n$ is given by
\begin{align}\label{cap}
\GG{W\cc n}(s)=\ee{(1-sW)^n}.
\end{align}
It turns out that under certain conditions on $W$ the r.h.s. of \eqref{cap} defines a proper \apgf even if $\supp{W}\cap(1,\infty)$ is nonnempty. In fact, we have the following.
\begin{theorem}\label{B_n}
Let $W$ be a nonnegative random variable and set $n\ge 1$. Then (\ref{cap}) is a proper \apgf if and only if $\ee{W^n}<\infty$ and $\ee{W^k(1-W)^{n-k}}\ge 0$ for every $k\in K_n$ where
\begin{equation}\label{eq:K_n}
K_n=\{k|\,k\in\II n, n-k\ \text{is odd}\}\,.
\end{equation}
\end{theorem}

\begin{remark}
Observe that for $n=1$ the conditions of Theorem~\ref{B_n} are met if and only if $\ee{W}\le 1$. For $n=2$ they are met if and only if $\ee{W^2}\le\ee{W}<\infty$ and from $(EW)^2\le EW^2$ this also implies that $EW\le 1$. Thus $\mathbb{B}(2)\subset \mathbb{B}(1)$ (also see Corollary~\ref{BCB} below). In both cases, the distribution of $W$ may have unbounded support; {\em e.g.,} $\exp(1)$ for $n=1$ and $\exp(2)$ for $n=2$ (also see Example~\ref{EXP} below).
\end{remark}

\begin{proof}
If (\ref{cap})  is a proper \apgf then in particular it is nonnegative and nonincreasing in $s$ and thus $\ee{(1-W)^n}\in[0,1]$. Since $\ee{(1-W)^n}1_{\{W\le 1\}}\in[0,1]$, then a separate consideration for $n$ odd and even implies that $\ee{(W-1)^n}1_{\{W> 1\}}\in[0,1]$.
Hence, $\ee{|W-1|^n}<\infty$, which is equivalent to $\ee{W^n}<\infty$. Obviously, since $W$ is nonnegative, $\ee W^k(1-W)^{n-k}\ge 0$ for $k\in K_n$ if and only if it holds for $k\in\II n$.

It therefore remains to prove the equivalence under the condition that $\ee{W^n}<\infty$. Since $\sum_{k=0}^n (1-s)^kp_k$ is an \apgf if and only if $p_k\ge 0$ for $k\in \II n$ and $\sum_{k=0}^np_k=1$ and since
\begin{equation}
\ee (1-sW)^n=\sum_{k=0}^n (1-s)^k{n\choose k}\ee W^k(1-W)^{n-k}\,,
\end{equation}
(in particular for $s=0$) the equivalence is immediate.
\end{proof}

From here on we will agree to write $W\in\BB$ whenever $W$ satisfies the conditions of Theorem~\ref{B_n}. In this case the notation $W\circ n$ denotes a random variable having a distribution with \apgf satisfying (\ref{cap}).

\begin{corollary}\label{BCB}
For every $n\ge 1$, $\BB\subset \mathbb{B}(n-1)$.
\end{corollary}

\begin{proof}
Let $W\in\BB$. It is easy to verify that for $k\in\II {n-1}$
\begin{equation}
\ee W^k(1-W)^{n-1-k}=\ee W^k(1-W)^{n-k}+\ee W^{k+1}(1-W)^{n-k-1}
\end{equation}
where, by Theorem~\ref{B_n}, both terms on the right are non-negative. Therefore the expression on the left is nonnegative for all $k\in\II {n-1}$ and in particular for $k\in K_{n-1}$. Applying Theorem~\ref{B_n} again we have that $W\in\mathbb{B}(n-1)$.
\end{proof}

\begin{remark}
It is natural to wonder whether there exists a random variable $W$ with $\supp{W}\cap(1,\infty)$ nonempty such that $W\in\BB$ for all $n\ge 1$. The answer is that this is impossible. The reason is that due to Theorem~\ref{B_n} this would be equivalent to $\ee{W^k(1-W)^\ell\ge0}$ for all $k,\ell$. From monotone convergence it would then follow that for odd $\ell$
\[
\ee{W^k(1-W)^\ell 1_{\{W>1\}}}\to -\infty
\]
as $k\to\infty$, while $\ee{W^k(1-W)^\ell 1_{\{W\le 1\}}}\le 1$ for all $k$. Thus, for sufficiently large $k$ we would necessarily have that $\ee{W^k(1-W)^\ell}<0$.
\end{remark}

\begin{remark}\label{unbdd}
Is it true that for every $n\ge 1$ there exists $W\in\BB$ with $\supp{W}\cap (1,\infty)$ nonnempty? Is it true that $\BB$ contains distributions with infinite support? 
Example~\ref{EXP}, to appear later, implies that the answer to both questions is yes.
\end{remark}

We would like to emphasize two facts. The first is that for $W\cc n$ an interpretation as a binomially mixed random variable is available only if $\supp{W}\subset [0,1]$. The second is that for any $W\in\BB$ the first $n$ moments of $W$ identify the distribution of $W\circ n$ and thus, in contrast to the case of binomial thinning, the distribution of $W$ is not uniquely determined from the distribution of $W\circ n$.

As for the $\cc$
 operation there are a number of immediate rules:
\begin{enumerate}
\item $(\al W)\cc n\deq \al\cc (W\cc n)$, $\alpha\in[0,1]$.
\item $\PP{W\cc n=k}={n\choose k}\ee{W^{k}(1-W)^{n-k}}$.
\item $\ee{{W\cc n\choose j}}={n\choose j}\ee{W^j}$ in particular $\ee{(W\cc n)}=n\ee{W}$. Recall ${k\choose j}=0$ when $k<j$.
\end{enumerate}
For a non-negative random variable $W$ with $\ee W^n<\infty$ we let
\begin{align*}
\sig W=\inf\{\al\in(0,\infty): W/\al\in\BB\}.
\end{align*}
In contrast to the often difficult evaluation of $\nor X$, it turns out that the value $\sig W$ is a root of a certain polynomial. Recall~(\ref{eq:K_n}).

\begin{theorem}\label{tig} When $\PP{W=0}=1$, $\sigma_n(W)=0$. Otherwise, for any $W$ with $\ee W^n<\infty$,
for each $k\in K_n$ there is a unique positive root $\alpha_k$ of the polynomial $g_k(\alpha)=\ee{W^k(\alpha-W)^{n-k}}$ and
\begin{equation}\label{eq:sigmaW}
\sigma_n(W)=\max\{\alpha_k|\,k\in K_n\}
\end{equation}
and $\rho(W/\sigma_n(W))\circ n)=1$. Consequently, for any $W\in\BB$ we have that $\nor{W\circ n}=\sig W$.
\end{theorem}

\begin{proof}
We first note that when $W=0$ a.s. then $(W/\alpha)\circ n=0$ a.s. for all $\alpha\in (0,\infty)$ (so that $W/\alpha\in\BB$) and thus $\sigma_n(W)=0$. As was discussed earlier, for this case we also clearly have that $\rho(W\circ n)=0$. Therefore, from here on it suffices to assume that $\PP{W=0}<1$.

From Theorem~\ref{B_n} we have, for $W\in\BB$ (hence, necessarily $\ee{W^n}<\infty$) and $\alpha\in (0,\infty)$, that $W/\alpha\in\BB$ if and only if $g_k(\alpha)\ge 0$ for $k\in K_n$, in which case 
\begin{equation}\label{eq:gk} 
\PP{(W/\alpha)\circ n=k}={n\choose k}\alpha^{-n}g_k(\alpha)\,.
\end{equation} 
For $k\in K_n$,  $g_k$ is continuous, strictly increasing with $g_k(0)=-\ee{W^k}<0$ and $\lim_{\alpha\to\infty} g_k(\alpha)=\infty$. Therefore there exists a unique positive $\alpha_k$ such that $g_k(\alpha_k)=0$. Consider 
\begin{equation}
k(W)=\arg\max\{\alpha_k|\,k\in K_n\}\,.
\end{equation} 
It is clear from Theorem~\ref{B_n} that for every $\alpha\in[\alpha_{k(W)},1]$ we have that
$W/\alpha\in\BB$ and that for every $\alpha\in[0,\alpha(W))$ we have that there exists some $k\in K_n$ such that $g_k(\alpha)<0$ so that by Theorem~\ref{B_n}, $W\not\in\BB$. Therefore $\sigma_n(W)=\alpha_{k(W)}$ and (\ref{eq:sigmaW}) is satisfied.

Since $g_{k(W)}(\alpha_{k(W)})=0$, $g_{k}(\alpha_{k(W)})\ge 0$ for $k\in K_n$ (since $\alpha_{k(W)}\ge \alpha_k$) and $g_n(\alpha_{k(W)})=\ee{W^n}>0$, it follows from (\ref{eq:gk}) that
the support of the distribution of $(W/\sigma_n(W))\circ n$ contains $n$ and has a hole in $k(W)$.  Thus, Theorem~\ref{hole} implies that $(W/\sigma_n(W))\circ n\in \MSF 1$. Since whenever $W\in\BB$ we have that $\sigma_n(W)\le 1$, then $W\circ n\sim \sigma_n(W)\circ((W/\sigma_n(W))\circ n)$ and it follows from Theorem~\ref{thmstar} that $W\circ n\in \MSF {\sigma_n(W)}$ and thus $\rho(W\circ n)=\sigma_n(W)$.
\end{proof}

\begin{remark}
In order to efficiently find the solutions of $EW^k(z-W)^{n-k}=0$ for $k\in K_n$, let us write  it  as
$a(z)=b(z)$, where
\begin{align}
a(z)=\ee{W^i(z-W)^{n-i}1_{\{W\le z\}}},\quad
b(z)=\ee{W^i(W-z)^{n-i}1_{\{W> z\}}}.\label{axbx}
\end{align}

The function $a(x)$ is non-negative and increasing with $a(0)=0$, while $b(x)$ is non-negative, decreasing with $b(0)>0$. Hence, in order to solve $a(x)=b(x)$ a binary search algorithm might be employed: start with an
initial value $z_0\in(0,1)$ and update according to
\begin{eqnarray*}
z_{n+1}=\begin{cases}
z_{n-1}+2^{-n-1}&\text{if}\quad a(z_{n-1})>b(z_{n-1})\\
z_{n-1}-2^{-n-1}&\text{if}\quad a(z_{n-1})<b(z_{n-1}).
\end{cases}
\end{eqnarray*}
\end{remark}

\begin{remark} 
If, for some $\alpha\in (0,1]$, $\supp{W}\subset [0,\alpha]$ then clearly $W/\alpha\in\BB$. Therefore, $\rho(W\circ n)=\sigma_n(W)\le \alpha$ and thus $W\circ n\in \MM \alpha$.
\end{remark}

The natural place of the following example is at the end of Section~\ref{DBS}. It appears here since Theorem~\ref{tig} allows a smoother derivation.

\begin{example}[Discrete uniform distribution]\label{DU} If $X\sim\text{Uniform}(\II n)$ then $\rho(X)=\frac{n}{n+1}$.

To see this, let $W\sim\Cunif 01$ and set $X\sim W\circ n$. Then, for $k\in\II n$,
\begin{align}
\PP{X=k}={n\choose k}\int_0^1 t^k(1-t)^{n-k}\,\text{d}t=\frac{1}{n+1},\label{sib}
\end{align}
so that indeed $X\sim\text{Uniform}(\II n)$. By Theorem~\ref{tig} $\nor{X}=\sig{W}$ is the maximum of the unique roots $\alpha_k$ of $g_k(\alpha)=\smallint_0^1t^k(\alpha-t)^{n-k}\,\text{d}t$ in $(0,1]$ for $k\in K_n$. When $k\ge 2$ (and $n-k$ is odd), integration by parts leads to
\begin{align*}
g_k(\alpha)&=\frac{k}{n-k+1}g_{k-1}(\alpha)-\frac{\kla{1-\alpha}^{n-k+1}}{n-k+1}
\\
&=\frac{k}{n-k+1}\frac{k-1}{n-k+2}g_{k-2}(\alpha)-\frac{\kla{1-\alpha}^{n-k+1}}{n-k+1}+\frac{k}{n-k+1}\frac{\kla{1-\alpha}^{n-k+2}}{n-k+2}.
\end{align*}
Since $g_k(\alpha_k)=0$ and $g_k(1)>0$ we have that $\alpha_k<1$. Therefore,
\begin{align*}
g_{k-2}(\alpha_k)
=\frac{1}{k-1}\kla{1-\alpha_k}^{n-k+1}\kla{\alpha_k+\frac{n+2}{k}}>0,
\end{align*}
so that necessarily $\alpha_k>\alpha_{k-2}$. Noting that $n-(n-1)=1$ is odd, it follows that $\alpha_{n-1}$ is the desired maximal root which determines $\rho(X)=\sigma_n(W)$. From
\begin{align*}
g_{n-1}(\alpha)=\int_0^1 t^{n-1}\kla{\alpha-t}\,\text{d}t=\frac{\alpha}{n}-\frac{1}{n+1},
\end{align*}
it follows that $\alpha_{n-1}=n/(n+1)$ and thus $\nor X=n/(n+1)$ or equivalently $X\in\MS{n/(n+1)}$.
\end{example}

\begin{remark}
When $X_n\sim \text{Uniform}(I_n)$ and $X_n\sim\frac{n}{n+1}\circ Z_n$, it is straightforward to show that 
(\ref{bj2}) with $\alpha=\frac{n}{n+1}$ becomes 
\begin{equation}
\PP{Z_n=j}=p^\star_j\left(\frac{n}{n+1}\right)=(n+1)^{j-1}\sum_{i=j}^n{i\choose j}\frac{1}{n^i}(-1)^{i-j}\,.
\end{equation}
In particular,
\begin{equation}
P(Z_n=0)=\frac{n^{n+1}-(-1)^{n+1}}{(n+1)^2n^n}\,,\quad P(Z_n=n-1)=0\,,\quad P(Z_n=n)=\frac{(n+1)^{n-1}}{n^n}\,,
\end{equation}
noting  that the distribution of $Z_n\in\MS 1$ has a hole in $n-1$.

Clearly, $\PP{Z_n=n}/\PP{X_n=n}$ is asymptotic to $e$ while $\PP{Z_n=0}/\PP{X_n=0}$ is asymptotic to one.
\end{remark}

The following theorem may be compared with Theorem~\ref{PoiApp}, where we  recall that $N(W)$ has a mixed Poisson distribution with conditional mean $W$.
\begin{theorem}\label{Wonto}
Suppose $W$ is a nonnegative random variable, $\seq an{\sq N}$ is a positive sequence such that $a_n/n\to 1$ as $n\to\infty$, $W/a_n\in\BB$ for each $n\ge 1$ and $\ee e^{\epsilon W}<\infty$ for some $\epsilon>0$.  Then
$(W/a_n)\circ n\dto N(W)$.
\end{theorem}

\begin{proof}
Since $a_n/n\to 1$, we have by bounded convergence that, for $s>0$,
\begin{equation}
\ee \left(1-s\frac{W}{a_n}\right)^n1_{\{W\le a_n/s\}}=\ee \left(\left(1-s\frac{W}{a_n}\right)^{a_n}\right)^{n/a_n}1_{\{W\le a_n/s\}}
\end{equation}
converges to $\ee e^{-sW}$, which is also equal to $G_{N(W)}(s)$ on $[0,1]$.
Now,
\begin{align}\label{gto0}
\ee \left(s\frac{W}{a_n}-1\right)^n1_{\{W>a_n/s\}}=e^{-\epsilon a_n/s}\kla{\frac{s}{a_n}}^n\ee \left(W-\frac{a_n}{s}\right)^ne^{-\epsilon(W-a_n/s)}e^{\epsilon W}1_{\{W>a_n/s\}}\,.
\end{align}
Since $x^ne^{-\epsilon x}$ achieves its maximum at $x=n/\epsilon$, this implies that the right side of (\ref{gto0}) is bounded above by
\begin{equation}
\kla{\frac{s}{a_n}}^n\kla{\frac{n}{\epsilon}}^n e^{-n-\epsilon a_n/s}  \ee e^{\epsilon W}
=\kla{\frac{ns}{\epsilon a_n}e^{-1-\frac{\epsilon a_n}{ns}}}^n  \ee e^{\epsilon W}=f\left(\frac{ns}{\epsilon a_n}\right)^n\ee e^{\epsilon W}\,,
\end{equation}
where $f(x)\equiv xe^{-1-x^{-1}}$ is continuous and strictly increasing on $(0,\infty)$.
Let $s^*\approx 3.59112$ be the unique positive solution of $f(x)=1$. For $s\in(0,\epsilon s^*)$ let $N_s$ be such that for $n\ge N_s$
\begin{equation}
\frac{n}{a_n}\le \frac{1+\epsilon s^*/s}{2}\,.
\end{equation}
Then,
\begin{equation}
f\left(\frac{ns}{\epsilon a_n}\right)\le f\left(\frac{s^*+s/\epsilon}{2}\right)<f(s^*)=1\,
\end{equation}
and thus $f(\frac{ns}{\epsilon a_n})^n\to 0$ as $n\to\infty$. Since $\ee e^{\epsilon W}<\infty$, it follows that the left hand side of (\ref{gto0}) vanishes as $n\to\infty$ on $[0,\epsilon s^*)$. All of the above implies that $G_{(W/a_n)\circ n}(s)\to G_{N(W)}(s)$ on some positive right neighborhood of $0$ from which the result follows.
\end{proof}

The next lemma will be needed for Example~\ref{EXP} to follow.
\begin{lemma}\label{noneg}
Let $W\sim\exp(1)$. For every odd integer $\ell\ge 1$ and real $\gamma\ge 0$ the root of
\begin{equation}
\ee W^\gamma(z-W)^\ell
\end{equation}
is positive and strictly smaller than $\gamma+\ell/2+1$. For $\ell=1$ the root is precisely $z=\gamma+1$.
\end{lemma}
\begin{proof}
For the first part it suffices to show that for  $z=\gamma+\ell/2+1$ we have that $\ee W^\gamma(z-W)^\ell>0$. Now,
\begin{align}
\ee W^\gamma(z-W)^\ell=\int_0^z w^\gamma (z-w)^\ell e^{-w}dw-\int_z^\infty w^\gamma(w-z)^\ell e^{-w}dw\,.
\end{align}
Dividing the right hand side by $z^{\gamma+\ell+1}$, making the change of variable $x=w/z$ in the first integral and $x=z/w$ in the second, gives
\begin{align}\label{i01}
\int_0^1(1-x)^\ell \left( x^\gamma e^{-zx}-\frac{1}{x^{\gamma+\ell+2}} e^{-z/x}\right)dx\,.
\end{align}
A sufficient condition for~(\ref{i01}) to be strictly positive is that the integrand is stricly positive on $(0,1)$ which is equivalent to
\begin{equation}
x^{2\gamma+\ell+2}e^{z(1/x-x)}>1\ .
\end{equation}
Noting that $(2\gamma+\ell+2)/z=2$ (since $z=\gamma+(\ell/2)+1$) this is equivalent to
\begin{equation}
f(x)=2\log x+x^{-1}-x>0\ .
\end{equation}
Since
\begin{equation}
f'(x)=\frac{2}{x}-\frac{1}{x^2}-1=-\left(\frac{1}{x}-1\right)^2\,,
\end{equation}
it follows that $f$ is strictly decreasing on $(0,1]$ with $f(1)=0$ and thus $f(x)>0$ for all $x\in(0,1]$.
To complete the derivation we observe that
\begin{equation}
\ee W^\gamma(z-W)=z\Gamma(\gamma+1)-\Gamma(\gamma+2)=(z-\gamma-1)\Gamma(\gamma+1)
\end{equation}
so that the root is clearly $z=\gamma+1$.
\end{proof}

The following example demonstrates the applicability of Theorems~\ref{tig} and~\ref{Wonto}. As promised in Remark~\ref{unbdd}, it also implies that for every $n\ge 1$ there are distributions in $\BB$ with unbounded support.

\begin{example}\label{EXP}[The gamma distribution] $W\sim\Cgam ab$ where $a,b>0$. Then $\sigma_n(W)=\frac{a+n-1}{b}$ for every $n\ge 1$. Moreover, $(W/\sigma_n(W))\circ n\dto  X$ where $X$ has a generalized negative binomial distribution with $a$ and $p=1/2$ (see~(\ref{negbin})). 

To see this, we first note that since $bW\sim\Cgam a1$ and $\sigma_n(bW)=b\sigma_n(W)$, it suffices to consider the case $b=1$. We observe that if $W\sim\Cgam a1$ and $X\sim\Cexp 1$ then
\begin{equation}\label{expgamma}
\ee W^k(z-W)^{n-k}=\frac{1}{\Gamma(a)}\ee X^{k+a-1}(z-X)^{n-k}\,.
\end{equation}
If $n\ge 1$ and $k=n-1$ then we have from Lemma~\ref{noneg}, with $\gamma=k+a-1=a+n-2$ and $\ell=n-k=1$, that the root of $\ee W^{n-1}(z-W)$ is $z=\gamma+1=a+n-1$. If $\gamma=k+a-1$ is such that $\ell=n-k$ is odd and greater than or equal to $3$ then from (\ref{expgamma}) and Lemma~\ref{noneg} it follows that the root of $\ee W^k(z-W)^{n-k}$ is strictly smaller than
$\gamma+\ell/2+1=k+a-1+\frac{n-k}{2}+1=a+\frac{k+n}{2}$. Since $n-k\ge 3$, it follows that $k\le n-3$ and thus $a+\frac{k+n}{2}\le a+n-\frac{3}{2}<a+n-1$. Therefore, if $\alpha_k$ are the roots of $\ee W^k(z-W)^{n-k}$ for $k\in K_n$, then all but one are strictly less than $a+n-1$ and one is equal to $a+n-1$. From Theorem~\ref{tig} it therefore follows that $\sigma_n(W)=a+n-1$.

Now, we recall that  $\ee e^{-sW}=G_{N(W)}(s)$ for $s\in [0,1]$. Also, from Example~\ref{negb}, when $W\sim\Cgam{a}{p/(1-p)}$, then $N(W)$ has a generalized negative binomial distribution with $a$ and $p$. Since $p/(1-p)=1$ then necessarily $p=1/2$.

Finally, for every $\epsilon\in (0,1)$ we have that $\ee e^{\epsilon W}=(\frac{1}{1-\epsilon})^a<\infty$ and, clearly, $\sigma_n(W)/n=1+\frac{a-1}{n}\to 1$ as $n\to\infty$. Thus the convergence result follows from Theorem~\ref{Wonto}.
\end{example}

\begin{remark}
Special cases of Example~\ref{EXP} are:
\begin{itemize}
\item $W\sim\Cexp\la$ with $\sigma_n(W)=n/\lambda$.
\item $W\sim\Cerl m\la$ (independent sum of $m$ $\Cexp \la$ distributed random variables) with $\sigma_n(W)=(m+n-1)/\lambda$.
\item $W\sim\chi^2_{(m)}$ (independent sum of squares of $m$ standard normally distributed random variables) with $\sigma_n(W)=m+2(n-1)$.
\end{itemize}
\end{remark}

\bibliographystyle{abbrvnat}
\bibliography{bibliography1}

\end{document}